\documentclass[11pt]{amsart}
\usepackage[utf8]{inputenc}
\usepackage{enumerate}
\usepackage{amsrefs}
\usepackage{dsfont}
\usepackage{amsmath}
\usepackage{amsthm}
\usepackage{amssymb}
\usepackage{amsfonts}
\usepackage{mathrsfs}  %to have \mathscr for script L
\usepackage[shortlabels]{enumitem} % more enumerate/itemize options
\usepackage{color}
\usepackage{xcolor}
\usepackage[bookmarks=true,hyperindex,pdftex,colorlinks,citecolor=blue]{hyperref}
\usepackage{marginnote} % add margin notes
\usepackage{graphicx} % to include graphics
\usepackage{caption} % options for table/figure caption size
\usepackage{soul} % strike through words
\usepackage{bm}
\usepackage[centertags]{amsmath}
\usepackage{tocvsec2}

\theoremstyle{plain}
\newtheorem{theorem}{Theorem}[section]
\newtheorem{lemma}[theorem]{Lemma}
\newtheorem{corollary}[theorem]{Corollary}
\newtheorem{proposition}[theorem]{Proposition}
\newtheorem{claim}{Claim}

\newtheorem{thmAlfa}{Theorem}

\theoremstyle{definition}
\newtheorem*{definition*}{Definition}
\newtheorem{definition}[theorem]{Definition}

\theoremstyle{remark}
\newtheorem{remark}[theorem]{Remark}

\newcommand{\ee}{\varepsilon}
\newcommand{\nn}{\mathbb{N}}

\newcommand{\co}{\text{cof}}

\newcommand{\eps}{\varepsilon}

\newcommand{\Kdb}{\mathbb K}

\newcommand{\Ndb}{\mathbb N}

\newcommand{\Rdb}{\mathbb R}

\newcommand{\cal}{\mathcal}

\newcommand{\supp}{\text{supp}\,}
\newcommand{\diam}{\text{diam}\,}

\begin{document}

%\allowdisplaybreaks

\title[]{Asymptotic smoothness and universality in Banach spaces}

\author{R.M.~Causey}
\address{R.M~Causey}
\email{rmcausey1701@gmail.com}

\author{G.~Lancien}
\address{G.~Lancien, Laboratoire de Math\'ematiques de Besan\c con, Universit\'e Bourgogne Franche-Comt\'e, 16 route de Gray, 25030 Besan\c con C\'edex, Besan\c con, France}
\email{gilles.lancien@univ-fcomte.fr}

\subjclass[2020]{Primary: 46B20. Secondary: 46B03, 46B06}
\keywords{Asymptotic smoothness in Banach spaces, universality, complexity}

\begin{abstract}  For $1<p\leqslant \infty$, we study the complexity and the existence of universal spaces for two classes of separable Banach spaces, denoted $\textsf{A}_p$ and $\textsf{N}_p$, and related to asymptotic smoothness in Banach spaces. We show that each of these classes is Borel in the class of separable Banach spaces. Then we build small families of Banach spaces that are both injectively and surjectively universal for these classes. Finally, we prove the optimality of this universality result, by proving in particular that none of these classes admits a universal space. 
\end{abstract}

\maketitle

\setcounter{tocdepth}{1}
%\tableofcontents

\section{Introduction}

The notion of asymptotic uniform smoothness has become very important in the
recent developments of the linear and non-linear geometry of Banach spaces. In a recent work \cite{CauseyLancienT_p} we have proved that the class $\textsf{T}_p \cap \textsf{Sep}$ of all separable Banach spaces admitting an equivalent $p$-asymptotically uniformly smooth norm is analytic non Borel in the class $\textsf{Sep}$ of all separable Banach spaces and that there exists a space $U_p \in \textsf{T}_p \cap \textsf{Sep}$ such that any space in $\textsf{T}_p \cap \textsf{Sep}$ is both isomorphic to a subspace and to a quotient of $U_p$. For a Banach space $X$, the property $\textsf{T}_p$ can be characterized in terms of some infinite game and in terms of the existence of upper-$\ell_p$-estimates for weakly null trees of infinite height in $B_X$, the unit ball of $X$ (see \cite{CauseyPositivity2018}).  

In Section \ref{properties}, we introduce the properties $\textsf{A}_p$ and $\textsf{N}_p$, which are both slightly weaker than  $\textsf{T}_p$ (and  $\textsf{N}_p$ is weaker than $\textsf{A}_p$). We give their definitions in terms of finite two players games and recall their main characterizations. First we give their characterizations in terms of  upper-$\ell_p$-estimates for weakly null trees of finite height in the unit ball. We also recall their dual characterizations. For that purpose we need to introduce the so-called Szlenk derivation in a dual Banach space and the associated notions of $q$-summable Szlenk index and convex Szlenk index for a Banach space. The aim of this paper is to address, for $\textsf{A}_p$ and $\textsf{N}_p$, the questions  solved for $\textsf{T}_p$ in \cite{CauseyLancienT_p}. 

In Section \ref{descriptive}, we  introduce the necessary framework, built by B. Bossard in \cite{Bossard2002}, for conducting a meaningful study of the topological complexity of a class of separable Banach spaces. Then, we use the dual characterizations of $\textsf{A}_p$ and $\textsf{N}_p$ to show the following. 

\begin{thmAlfa} Let $p\in (1,\infty]$. The classes $\emph{\textsf{Sep}} \cap \emph{\textsf{A}}_p$ and $\emph{\textsf{Sep}} \cap \emph{\textsf{N}}_p$ are Borel.
\end{thmAlfa}

Next, we start our construction of universal families for ${\textsf{Sep}} \cap {\textsf{A}}_p$ and ${\textsf{Sep}} \cap {\textsf{N}}_p$. This will take a few steps. In Subsection \ref{modelspaces}, we build a first family of model spaces for $\textsf{A}_p$ and $\textsf{N}_p$. If $q$ is the conjugate exponent of $p$ and $\theta \in (0,1)$, we denote $T_{q,\theta}^*$ the dual of the $q$-convexification of $T_\theta$, the Tsirelson space of parameter $\theta$, and we show that $T_{q,\theta}^* \in \textsf{A}_p$. We modify slightly the construction to get a typical $\textsf{N}_p$ space $U_{q,\theta}^*$. In fact to complete this family, we need to introduce another parameter $M$, infinite sequence in $\Ndb$, and associated spaces $T_{q,\theta,M}^*$ and $U_{q,\theta,M}^*$. In Subsection \ref{pressdown} we introduce the key notion of pressdown norm associated with a Banach space $Z$ with finite dimensional decomposition $\textsf{E}$ and a Banach space $T$ with a $1$-unconditional basis. The associated Banach space is denoted $Z_\wedge^T(\textsf{E})$ and we show that if $T$ has $\textsf{A}_p$ (resp. $\textsf{N}_p$), then $Z_\wedge^T(\textsf{E})$ has $\textsf{A}_p$ (resp. $\textsf{N}_p$). Then, in Subsection \ref{gliding}, we gather technical results about the interaction of gliding hump arguments and quotient maps. 

In Section \ref{FDD's}, we take a crucial step, by showing that for any space $X$ in ${\textsf{Sep}} \cap {\textsf{A}}_p$ (resp. in ${\textsf{Sep}} \cap {\textsf{N}}_p$), there exist $\theta\in (0,1)$ and  Banach spaces $Z,Y$ with FDD's $\textsf{{F}}$, $\textsf{{H}}$, respectively, such that $X$ is isomorphic to a subspace of $Z^{T^*_{q,\theta}}_\wedge(\textsf{{F}})$, and to a quotient of $Y^{T^*_{q,\theta}}_\wedge(\textsf{{H}})$ (resp. to a subspace of $Z^{U^*_{q,\theta}}_\wedge(\textsf{{F}})$, and to a quotient of $Y^{U^*_{q,\theta}}_\wedge(\textsf{{H}})$).

In Section \ref{universal} we take the final step of our construction, which is to use the complementably universal space for Banach spaces with an FDD built by Schechtman in \cite{Schechtman1975}. We denote $W$ this universal space and $\textsf{J}$ its finite dimensional decomposition. Then we show:

\begin{thmAlfa} Fix $1<p\leqslant \infty$ and let $q$ be its conjugate exponent. Let $X$ be a separable Banach space. Then
\begin{enumerate}[(i)]
\item $X$ has $\textsf{\emph{A}}_p$ if and only if there exist $\theta\in (0,1)$ and an infinite sequence $M$ in $\Ndb$ such that $X$ is isomorphic to a subspace of $W^{T^*_{q,\theta,M}}_\wedge(\textsf{\emph{J}})$ if and only if there exist $\theta\in (0,1)$ and an infinite sequence $M$ in $\Ndb$ such that $X$ is isomorphic to a quotient of $W^{T^*_{q,\theta,M}}_\wedge(\textsf{\emph{J}})$.
\item $X$ has $\textsf{\emph{N}}_p$ if and only if there exist $\theta\in (0,1)$ and and an infinite sequence $M$ in $\Ndb$ such that $X$ is isomorphic to a subspace of $W^{U^*_{q,\theta,M}}_\wedge(\textsf{\emph{J}})$ if and only if there exist $\theta\in (0,1)$ and an infinite sequence $M$ in $\Ndb$ such that $X$ is isomorphic to a quotient of $W^{U^*_{q,\theta,M}}_\wedge(\textsf{\emph{J}})$.
\end{enumerate}
\end{thmAlfa}

In the concluding Section \ref{optimal}, we show that this result is optimal. More precisely, we introduce yet another two players game to show the following. 

\begin{thmAlfa}
Fix $1<p\leqslant \infty$. If $U$ is any Banach space with $\textsf{\emph{N}}_p$, then there exists a Banach space $X$ with $\textsf{\emph{A}}_p$ such that $X$ is not isomorphic to any subspace of any quotient of $U$. More precisely, if $q$ is the conjugate exponent of $p$, then there exists $\theta \in (0,1)$ such that $T_{q,\theta}^*$ is not isomorphic to any subspace of any quotient of $U$.
\end{thmAlfa}

\section{The properties and their characterizations}\label{properties}

All Banach spaces are over the field $\mathbb{K}$, which is either $\mathbb{R}$ or $\mathbb{C}$.   We denote $B_X$ (resp. $S_X$) the closed unit ball (resp. sphere) of a Banach space $X$.    By \emph{subspace}, we shall always mean closed subspace.   Unless otherwise specified, all spaces are assumed to be infinite dimensional. For a Banach space $X$, we denote $W_X$ the set of weak open neighborhoods of $0$ in $X$, by $\co(X)$ the set of closed finite-codimensional subspaces of $X$ and by $K_X$ the set of norm compact subsets of $B_X$. For $X,Y$ Banach spaces, a bounded linear map $Q:X \to Y$ is called a \emph{quotient map} if it is onto and induces an isometry from $X/\ker(Q)$ onto $Y$. Let $X$ and $Y$ be two isomorphic Banach spaces, the \emph{Banach-Mazur distance} from $X$ to $Y$ is $d_{BM}(X,Y)=\inf\{\|T\|\,\|T^{-1}\|,\ T\ \text{isomorphism from}\ X\ \text{onto}\ Y\}$. 

Let $D$ be a set. We denote $D^{\leqslant n}=\cup_{i=1}^n D^i$, $D^{<\omega}=\cup_{i=1}^\infty D^i$,  $D^\omega$ the set of all infinite sequences whose members lie in $D$ and $D^{\leqslant \omega}=D^{<\omega}\cup D^\omega$.  For $s,t\in D^{<\omega}$, we let $s\smallfrown t$ denote the concatenation of $s$ with $t$. We let $|t|$ denote the length of $t$.  For $0\leqslant i\leqslant|t|$, we let $t|_i$ denote the initial segment of $t$ having length $i$, where $t|_0=\varnothing$ is the empty sequence.  If $s\in \{\varnothing\}\cup D^{<\omega}$, we let $s\prec t$ denote the relation that $s$ is a proper initial segment of $t$.

We start with the definition of the Szlenk index.  For a Banach space $X$, $K\subset X^*$ weak$^*$-compact, and $\ee>0$, we let $s_\ee(K)$ denote the set of $x^*\in K$ such that for each weak$^*$-neighborhood $V$ of $x^*$, $\text{diam}(V\cap K)\ge \ee$.    We define the transfinite derivations 
\[s_\ee^0(K)=K,\ \ s^{\xi+1}_\ee(K)=s_\ee(s^\xi_\ee(K)),\] and if $\xi$ is a limit ordinal, \[s^\xi_\ee(K)=\bigcap_{\zeta<\xi}s_\ee^\zeta(K).\]  
For convenience, we let $s_0(K)=K$.     If there exists an ordinal $\xi$ such that $s^\xi_\ee(K)=\varnothing$, we let $Sz(K,\ee)$ denote the minimum such ordinal, and otherwise we write $Sz(K,\ee)=\infty$.   We let $Sz(K)=\sup_{\ee>0} Sz(K,\ee)$, where $Sz(K)=\infty$ if $Sz(K,\ee)=\infty$ for some $\ee>0$.   We let $Sz(X,\ee)=Sz(B_{X^*},\ee)$ and $Sz(X)=Sz(B_{X^*})$. In this work, we will exclusively be concerned with Banach spaces $X$ such that $Sz(X)\leqslant \omega$, where $\omega$ is the first infinite ordinal.  By compactness, $Sz(X)\leqslant \omega$ if and only if $Sz(X,\ee)$ is a natural number for each $\ee>0$. We recall that $Sz(X)<\infty$ if and only if $X$ is \emph{Asplund} (see \cite{Lancien2006} and references therein). One characterization of Asplund spaces is that every separable subspace has a separable dual.

For $1\leqslant q<\infty$, we say $X$ has $q$-\emph{summable Szlenk index} provided there exists a constant $c>0$ such that for any $n\in\nn$ and any $\ee_1, \ldots, \ee_n \geqslant 0$ such that $s_{\ee_1}\ldots s_{\ee_n}(B_{X^*})\neq \varnothing$, $\sum_{i=1}^n \ee_i^q \leqslant c^q$.   In the $q=1$ case, we refer to this as \emph{summable Szlenk index} rather than $1$-summable Szlenk index. 

We shall also need a somewhat slower derivation and the corresponding index, called the \emph{convex Szlenk index} $Cz(X)$, introduced in \cite{GKL2001}, which is defined identically from the following derivation: for $K\subset X^*$ weak$^*$-compact, and $\ee>0$, $c_\eps(K)$ is the weak$^*$-closed convex hull of $s_\eps(K)$.

\medskip
We now define the asymptotic smoothness properties that we shall study by the means of two different two-players games on a Banach space $X$ (and their variants). Fix $1<p\leqslant \infty$ and let $q$ be its conjugate exponent. For $n \in \Ndb$, we denote $\ell_p^n$ the space $\Kdb^n$ equipped with the $p$-norm: $\|a\|_p=(\sum_{i=1}^n|a_i|^p)^{1/p}$, $a\in \Kdb^n$. For $c>0$ and $n\in\nn$, we define the $A(c,p,n)$ game and the $N(c,p,n)$ game.  In the $A(c,p,n)$ game, Players I and II take turns choosing $U_i\in W_X$ and $x_i\in U_i\cap B_X$, respectively, until $(x_i)_{i=1}^n$ has been chosen.  Player I wins if \[\max\Bigl\{\Bigl\|\sum_{i=1}^n a_ix_i\Bigr\|:(a_i)_{i=1}^n\in B_{\ell_p^n}\Bigr\}\leqslant c,\] and Player II wins otherwise. The \emph{spatial} $A(c,p,n)$ game on $X$ is similar, except Player I chooses $Y_i\in \co(X)$ and Player II chooses $x_i\in B_{Y_i}$. The conditions for Player I or Player II winning are the same as in the $A(c,p,n)$ game. The \emph{compact spatial} game is also similar, except Player I chooses $Y_i\in \co(X)$ and Player II chooses $C_i\in K_{Y_i}$, where for $Y\in \co(X)$, $K_Y$ denotes the set of norm-compact subsets of $B_Y$.  Player I wins if \[\max\Bigl\{\Bigl\|\sum_{i=1}^n a_ix_i\Bigr\|:(a_i)_{i=1}^n \in B_{\ell_p^n}, (x_i)_{i=1}^n \in \prod_{i=1}^n C_i\Bigr\}\leqslant c,\] and Player II wins otherwise. 

The $N(c,p,n)$ game is similar to the $A(c,p,n)$ game. Only the winning condition is modified. Player I wins if $\Bigl\|\sum_{i=1}^n x_i\Bigr\|\leqslant c n^{1/p},$ and Player II wins otherwise. The modifications needed to define the \emph{spatial} $N(c,p,n)$ and the \emph{compact spatial} $N(c,p,n)$ game are identical to those for the $A(c,p,n)$ games. 

Let us now precise what we mean by strategies in these games. We define only the notions of strategies and winning strategies for Player I. For a Banach space $X$ and $n\in\nn$, an $n$-\emph{strategy}  is a function $\chi:B_X^{<n}\to W_X$.    A \emph{spatial} $n$-\emph{strategy} is a function $\chi:B_X^{<n}\to \co(X)$. A \emph{compact spatial} $n$-\emph{strategy} is a function $\chi:K_X^{<n}\to \co(X)$. If $\chi$ is an $n$-strategy, then we say $(x_i)_{i=1}^n\subset B_X$ is $\chi$-\emph{admissible} if $x_j\in \chi((x_i)_{i=1}^{j-1})$ for all $1\leqslant j\leqslant n$.   The notion of $\chi$-admissibility for a spatial $n$-strategy is defined similarly. If $\chi$ is a compact spatial $n$-strategy, we say $(C_i)_{i=1}^n \in K_X^n$ is $\chi$-admissible if $C_j\subset \chi((C_i)_{i=1}^{j-1})$ for all $1\leqslant j\leqslant n$.  For any type of strategy in any of the games defined above, we say the strategy is a \emph{winning strategy} if any sequence admissible with respect to it satisfies the winning condition of the game for Player I. 

It is known (see \cite{CauseyPositivity2018}, Section 3) that each of these games is determined. That is, in each game, either Player I or Player II has a winning strategy. We let $\textsf{a}_{p,n}(X)$ denote the infimum of $c>0$ such that Player I has a winning strategy in the $A(c,p,n)$ game, and we let $\textsf{a}_p(X)=\sup_n \textsf{a}_{p,n}(X)$. We note that $\textsf{a}_p(X)$ is the infimum of $c>0$ such that for each $n\in\nn$, Player I has a winning strategy in the $A(c,p,n)$ game if such a $c$ exists, and $\textsf{a}_p(X)=\infty$ otherwise. We let $\textsf{n}_{p,n}(X)$ denote the infimum of $c>0$ such that Player I has a winning strategy in the $N(c,p,n)$ game, and $\textsf{n}_p(X)=\sup_n \textsf{n}_{p,n}(X)$. 

We shall also use the following infinite game. First, denote $c_{00}$ the space of all finitely supported scalar sequences. Then, for two sequences $(e_n)_{n=1}^\infty$, $(f_n)_{n=1}^\infty$ in (possibly different) Banach spaces and for $c>0$, we write $(e_n)_{n=1}^\infty \lesssim_c (f_n)_{n=1}^\infty$ provided that 
$$\forall (a_n)_{n=1}^\infty\in c_{00}\ \ \ \Bigl\|\sum_{n=1}^\infty a_ne_n\Bigr\| \leqslant c\Bigl\|\sum_{n=1}^\infty a_nf_n\Bigr\|.$$ 
For a Banach space $T$ with basis $(e_i)_{i=1}^\infty$ and $c>0$, we define the \emph{spatial} $(T,c)$ game on $X$. Players I and II take turns choosing $Y_i\in \co(X)$ and $x_i\in B_{Y_i}$, respectively.  Player I wins if $(x_i)_{i=1}^\infty\lesssim_c (e_i)_{i=1}^\infty$, and Player II wins otherwise. The notions of $\omega$-strategies, admissibility and $\omega$-winning strategies are defined identically. 

\begin{definition} Let $p\in (1,\infty]$. 
We define  $\textsf{A}_p$ to be the class of all Banach spaces $X$ such that $\textsf{a}_p(X)<\infty$
and $\textsf{N}_p$ the class of all Banach spaces $X$ such that $\textsf{n}_p(X)<\infty$. 
\end{definition}

The following proposition relies on routine approximation arguments.

\begin{proposition} Fix $1<p\leqslant \infty$ and let $X$ be a Banach space.   
\begin{enumerate}[(i)]
\item $X$ has $\textsf{\emph{A}}_p$ if and only if there exists $c>0$ such that for all $n\in\nn$, Player I has a winning strategy in the spatial $A(c,p,n)$ game if and only if there exists $c>0$ such that for all $n\in\nn$, Player I has a winning strategy in the compact spatial $A(c,p,n)$ game. 
\item $X$ has $\textsf{\emph{N}}_p$ if and only if there exists $c>0$ such that for all $n\in\nn$, Player I has a winning strategy in the spatial $N(c,p,n)$ game if and only if there exists $c>0$ such that for all $n\in\nn$, Player I has a winning strategy in the compact spatial $N(c,p,n)$ game.
\end{enumerate}
\end{proposition}

We now recall the main characterizations of these classes. We refer to \cite{CauseyFovelleLancien} for an overview of these properties and complete references. The results stated here come from \cite{CauseyIllinois2018} and \cite{Causey3.5}. Before to give these characterizations, we need more notation. Given $D$ a weak neighborhood base of $0$ in $X$ and $(x_t)_{t\in D^{<\omega}}\subset X$, we say $(x_t)_{t\in D^{<\omega}}$ is  \emph{weakly null}   provided that for each $t\in \{\varnothing\}\cup D^{<\omega}$, $(x_{t\smallfrown (U)})_{U\in D}$ is a weakly null net. Here $D$ is directed by reverse inclusion. 

\begin{theorem}[\cite{CauseyIllinois2018}]\label{atheorem} 
Fix $1<p\leqslant \infty$ and let $q$ be conjugate to $p$. Let $X$ be a Banach space. The following are equivalent  
\begin{enumerate}[(i)]
\item $X\in \textsf{\emph{A}}_p$. 
\item There exists a constant $c>0$ such that for any weak neighborhood base $D$ at $0$ in $X$, any $n\in\nn$,  and any weakly null collection $(x_t)_{t\in D^{\leqslant n}}\subset B_X$, there exists $t\in D^n$ such that $\|\sum_{i=1}^na_ix_{t|_i}\|\leqslant c\|a\|_p$ for all $a\in \Kdb^n$. 
\item $X$ has $q$-summable Szlenk index.   
\end{enumerate}
\end{theorem}

\begin{theorem}[\cite{Causey3.5}]\label{ntheorem} 
Fix $1<p\leqslant \infty$ and let $q$ be conjugate to $p$. Let $X$ be a Banach space. The following are equivalent. 
\begin{enumerate}[(i)]
\item $X\in \textsf{\emph{N}}_p$. 
\item There  exists a constant $c>0$ such that for any $n\in\nn$ and any weakly null collection  $(x_t)_{t\in D^{\le n}}$ in $B_X$, there exists $t\in D^n$ such that $\|\sum_{i=1}^nx_{t|_i}\|\le cn^{1/p}$.
\item There exists a constant $K>0$ such that 
$$\forall \eps \in (0,1),\ \ Cz(X,\eps)\leqslant {K}{\eps^{-q}}.$$
\end{enumerate}
\end{theorem}

Denote $\textsf{{D}}_1$ the class of all Banach spaces with Szlenk index at most $\omega$ and $\textsf{{T}}_p$ the ``infinite game version'' of $\textsf{{A}}_p$ (see \cite{CauseyFovelleLancien} for the precise definition). We recall the following inclusions.  
\begin{theorem}\label{containments}\  
\begin{enumerate}[(i)]
\item $\textsf{\emph{D}}_1=\bigcup_{1<p\leqslant \infty}\textsf{\emph{T}}_p=\bigcup_{1<p\leqslant \infty}\textsf{\emph{A}}_p=\bigcup_{1<p\leqslant \infty}\textsf{\emph{N}}_p$. 
\item For $1<p<\infty$, $\textsf{\emph{T}}_p\subsetneq \textsf{\emph{A}}_p\subsetneq \textsf{\emph{N}}_p$. 
\item $\textsf{\emph{T}}_\infty\subsetneq \textsf{\emph{A}}_\infty= \textsf{\emph{N}}_\infty$.
\end{enumerate}
\end{theorem}

\section{Descriptive Set Theory and asymptotic smoothness}\label{descriptive}

\subsection{Background} We recall the setting introduced by B. Bossard in \cite{Bossard2002} in order to apply the tools from descriptive set theory to the class $\textsf{Sep}$ of separable Banach spaces. 

A \emph{Polish space} (resp. \emph{topology}) is a separable completely metrizable space (resp. topology). A set $X$ equipped with a $\sigma$-algebra is called a \emph{standard Borel space} if the $\sigma$-algebra is generated by a Polish topology on $X$. A subset of such a  standard Borel space $X$ is called \emph{Borel} if it is an element of the $\sigma$-algebra and it is called \emph{analytic} (or a $\Sigma_1^1$ set) if there exists a standard Borel space $Y$ and a Borel subset $B$ of $X \times Y$ such that $A$ is the projection of $B$ on the first coordinate. The complement of an analytic set is called a \emph{coanalytic} set (or a $\Pi_1^1$ set). A subset $A$ of standard Borel space $X$ is called \emph{$\Sigma_1^1$-hard} if for every $\Sigma_1^1$ subset $B$ of a standard Borel space $Y$, there exists a Borel map $f:Y \to X$ such that $f^{-1}(A)=B$ and it is called \emph{$\Sigma_1^1$-complete} if it is both $\Sigma_1^1$ and $\Sigma_1^1$-hard. 

Let $X$ be a Polish space. Then, the set $\cal F(X)$ of all closed subsets of $X$ can be equipped with its \emph{Effros-Borel structure}, defined as the $\sigma$-algebra generated by the sets $\{F \in \cal F(X),\ F\cap U\neq \emptyset\}$, where $U$ varies over the open subsets of $X$. Equipped with this $\sigma$-algebra, $\cal F(X)$ is a standard Borel space. 

Following Bossard, we now introduce the fundamental coding of separable Banach spaces. It is well known that $C(\Delta)$, the space of scalar valued continuous functions on the Cantor space $\Delta=\{0,1\}^\Ndb$, equipped with the sup-norm, contains an isometric linear copy of every separable Banach space. We equip $\cal F(C(\Delta))$ with its corresponding Effros-Borel structure. Then, we denote 
$$\textsf{{SB}}=\{F\in \cal F(C(\Delta)),\ F\ \text{is a linear subspace of}\ C(\Delta)\},$$
considered as a subspace of  $\cal F(C(\Delta))$. Then $\textsf{{SB}}$ is a Borel subset of $\cal F(C(\Delta))$ (\cite{Bossard2002}, Proposition 2.2) and therefore a standard Borel space, that we call the \emph{standard Borel space of separable Banach spaces}. 

Let us now denote $\simeq$ the isomorphism equivalence relation on $\textsf{{SB}}$. The \emph{fundamental coding of separable Banach spaces} is the quotient map $c:\textsf{{SB}} \to \textsf{{SB}}/\simeq$. We can now give the following definition.

\begin{definition} A family $\textsf{{G}}\subset \textsf{{SB}}/\simeq$ is Borel (resp. analytic, coanalytic) if $c^{-1}(\textsf{{G}})$  is Borel (resp. analytic, coanalytic) in $\textsf{{SB}}$. 
\end{definition}
This will allow us to describe the complexity of classes of separable Banach spaces that are stable under linear isomorphisms, such as $\textsf{T}_p$, $\textsf{A}_p$ and $\textsf{N}_p$. It will sometimes be convenient to use another coding of separable Banach spaces, by using the fact that any separable Banach space is a quotient of $\ell_1$. For a sequence $\overline{x}=(x_n)_{n=1}^\infty \in \ell_1^\omega$, define $c_d(\overline{x})=\langle \ell_1/\overline{sp}(\overline{x})\rangle$ (where $\langle E \rangle$ denotes the equivalence class of a separable Banach space in $\textsf{{SB}}/\simeq$). The following is taken from \cite{Bossard2002} (Proposition 2.8).

\begin{proposition}\label{equivalentcoding}
A family $\textsf{\emph{G}}\subset \textsf{\emph{SB}}/\simeq$ is Borel (resp. analytic, coanalytic) if and only if $c_d^{-1}(\textsf{\emph{G}})$  is Borel (resp. analytic, coanalytic) in $\ell_1^\omega$. \end{proposition}

Let $K$ be the closed unit ball of $\ell_\infty$ equipped with the weak$^*$-topology induced by $\ell_1$. Then $K$ is a metrizable compact space. We denote $\cal F(K)$ the set of closed subsets of $K$. The \emph{Vietoris topology} on $\cal F(K)$ is the topology generated by the sets of the form $\{F\in \cal F(K),\ F\subset O\}$ and $\{F\in \cal F(K),\ F\cap O \neq \varnothing \}$, for $O$ open subset of $K$. It is a Polish topology as it is compact and metrizable. Then the Borel $\sigma$-algebra associated with this topology is generated by the sets $\{F\in \cal F(K),\ F\cap O \neq \varnothing \}$, for $O$ open subset of $K$. It can also be described as the $\sigma$-algebra generated by $\{F\in \cal F(K),\ F \subset O\}$, for $O$ open subset of $K$. We shall also use the following (\cite{Bossard2002}, Lemma 4.14).

\begin{proposition}\label{F(K)coding} Define $k:\ell_1^\omega \to \cal F(K)$ so that, for $\overline{x}\in \ell_1^\omega$, $k(\overline{x})$ is the closed unit ball of the orthogonal of the linear span of $\overline{x}$ in $\ell_1$. Then $k$ is Borel.
\end{proposition}

\subsection{The classes $\textsf{Sep}\cap \textsf{A}_p$ and $\textsf{Sep}\cap \textsf{N}_p$ are Borel}

For $\eps>0$, we consider the following derivations on $\cal F(K)$. For $F\in \cal F(K)$, 

\begin{enumerate}[(i)]
    \item $s_\eps(F)=F\setminus \cup\{V,\ V\ \text{weak}^*\text{open set so that}\ \diam(S\cap F)\le \eps\},$
    \item $c_\eps(F)\ \text{is the weak$^*$ closed convex hull of $s_\eps(F)$},$
    \item $k_\eps(F)=F\setminus \cup\{S,\ S\ \text{weak}^*\text{open half space so that}\ \alpha(S\cap F)\le \eps\},$ where $\alpha(A)$ is the Kuratowski index of $A\subset K$ defined by 
$$\alpha(A)=\sup\{\delta\ge 0,\ \forall n\in \Ndb\ \exists x_1^*,\ldots,x_n^* \in A\ \|x_i^*-x_j^*\|\geqslant \delta\ \text{for}\ i\neq j\}.$$
\end{enumerate}
Associated with the derivations $s_\eps$, $c_\eps$ and $k_\eps$ we define $Sz(F)$, the Szlenk index of $F$, $Cz(F)$, the convex Szlenk index of $F$ and $Kz(F)$. We recall the following estimates from \cite{HajekLancien2007}.

\begin{proposition}\label{Cz=Kz} For any $F$ weak$^*$ closed convex subset of $K$ and any $\eps >0$ we have that 
$$c_{4\eps}(F)\subset k_{2\eps}(F)\subset c_\eps(F).$$
\end{proposition}

For the derivation $s_\eps$, the following statement is due to B. Bossard. Its proof can be found in his PhD thesis \cite{Bossard1994}, but unfortunately not in his paper \cite{Bossard2002}. We detail here an adaptation of this proof for the derivation $k_\eps$.

\begin{proposition}\label{Borelderivations} Let $\eps>0$. Then, the maps $s_\eps:\cal F(K) \to \cal F(K)$ and $k_\eps:\cal F(K) \to \cal F(K)$ are Borel. 
\end{proposition}

\begin{proof}
First we fix a dense sequence $(x_n)_{n=1}^\infty$ in $\ell_1$ and a dense sequence $(r_k)_{k=1}^\infty$ in $\Rdb$ and we denote by $\cal S$ the countable set of weak$^*$ open slices of the form $\{\text{Re}\,x_n>r_k\}$ or $\{\text{Re}\,x_n<r_k\}$. We also fix a countable base $\cal V$ of open sets for the topology of $K$. For $x^*\in K$, we denote $\cal S(x^*)$ the set of elements of $\cal S$ containing $x^*$ and $\cal V(x^*)$ the set of elements of $\cal V$ containing $x^*$.

Fix now $O$ an open subset of $K$. We need to show that $H(O)=\{F \in \cal F(K),\ k_\eps(F)\subset O\}$ is Borel. First note that an easy approximation argument implies that
\begin{align*}
H(O)=\{F\in \cal F(K),\ \forall x^*\in F\setminus O\ \exists S\in \cal S(x^*)\ \alpha(S\cap F)\leqslant \eps\}.
\end{align*}
Then, using weak$^*$-compactness, we deduce that 
$$H(O)=\bigcup_{I\in \cal P_f(\cal S)}\Big[\Big\{F\in \cal F(K),\ F\subset (\bigcup_{S\in I}S)\cup O\Big\}\cap \bigcap_{S\in I}\Big\{F \in \cal F(K),\ \alpha(S\cap F)\leqslant \eps\Big\}\Big],$$ 
where $\cal P_f(\cal S)$ denotes the set of finite subsets of $\cal S$. This set being countable and the sets $\{F\in \cal F(K),\ F\subset (\bigcup_{S\in I}S)\cup O\}$ being open, we only need to show the following lemma.

\begin{lemma} For any $S\in \cal S$ and any $\eps>0$, the set $\{F \in \cal F(K),\ \alpha(S\cap F)\leqslant \eps\}$ is Borel.
\end{lemma}

Let us fix $S\in \cal S$ and $\eps>0$. Noting that $\alpha(S \cap F)>\eps$ if and only if there exists $\delta>\eps$ such that for all $n\in \Ndb$, there exists $x_1^*,\ldots,x_n^*$ in $S\cap F$ such that $\|x_i^*-x_j^*\|> \delta$ for all $i\neq j$. We only need to show that for all $n\in \Ndb$, the set 
$$A_n=\{F\in \cal F(K),\ \exists x_1^*,\ldots,x_n^* \in S\cap F\ \|x_i^*-x_j^*\|> \delta\ \text{for all}\  i\neq j\}$$
is Borel. 

It is easy to see that if $\|x_i^*-x_j^*\|> \delta$ for all $i\neq j$, then there exist $V_1,\ldots,V_n$ so that for all $i\leqslant n$, $V_i\in \cal V(x_i^*)$ and for all $i \neq j$ and all $(x^*,y^*)\in V_i\times V_j$, $\|x^*-y^*\|>\delta$. So let us denote 
$$\cal W_n=\{(V_1\times \cdots \times V_n)\in \cal V^n,\ \forall (x_1^*,\ldots,x_n^*)\in (V_1\times \cdots \times V_n)\ \forall i\neq j\ \|y_i^*-y_j^*\|> \delta\}.$$
Then we can write 
$$A_n=\bigcup_{(V_1\times \cdots \times V_n)\in \cal W_n} \bigcap_{i=1}^n \{F\in \cal F(K),\ F\cap V_i \cap S \neq \emptyset\}.$$
This shows that $A_n$ is open in $\cal F(K)$ and finishes the proof.

\end{proof}

We can now  prove our regularity result on the classes $\textsf{A}_p$ and $\textsf{N}_p$.

\begin{theorem} Let $p\in (1,\infty]$. The classes $\emph{\textsf{Sep}} \cap \emph{\textsf{A}}_p$ and $\emph{\textsf{Sep}} \cap \emph{\textsf{N}}_p$ are Borel.
\end{theorem}

\begin{proof} It follows from Proposition \ref{Borelderivations} that for fixed $\eps_1,\ldots,\eps_n>0$, the set 
$$\{F\in \cal F(K),\ s_{\eps_1}\ldots s_{\eps_n}(F) \neq \emptyset\}$$ 
is Borel. Let $q$ be the conjugate exponent of $p$. We have that $F$ has $q$-summable Szlenk index if and only if  
$$F \in \bigcup_{c\ge1}\bigcap_{\sum_{i=1}^n \eps_i^q>c}\{G\in \cal F(K), s_{\eps_1}\ldots s_{\eps_n}(G) = \emptyset\}.$$
Of course this union and this intersection can be taken countable, so that the condition ``having a $q$-summable Szlenk index'' is Borel. We can now apply Propositions \ref{equivalentcoding} and \ref{F(K)coding} to deduce that the class of separable Banach spaces with $q$-summable Szlenk index is Borel. Finally we use the fact that spaces in $\textsf{A}_p$ are exactly those with $q$-summable Szlenk index (Theorem \ref{atheorem}), to conclude that the class $\textsf{Sep} \cap \textsf{A}_p$ is Borel. 

Similarly, we deduce from Proposition \ref{Borelderivations} that the set of all $F\in \cal F(K)$ such that there exists $C\ge 1$ so that $Kz(F,\eps)\leqslant C\eps^{-p}$ for all $\eps\in (0,1)$ is a Borel subset of $\cal F(K)$. We recall (Theorem \ref{ntheorem}) that a Banach space is in $\textsf{N}_p$ if and and only if its convex Szlenk index is of power type $p$. So we conclude, applying Propositions \ref{equivalentcoding}, \ref{F(K)coding} and \ref{Cz=Kz} that the class $\textsf{Sep} \cap \textsf{N}_p$ is Borel. 

\end{proof}

\section{First tools}

\subsection{Model spaces}\label{modelspaces} In this subsection, we introduce the fundamental spaces that we will use to build our universal families for $\textsf{{A}}_p \cap \textsf{{Sep}}$ and $\textsf{{N}}_p \cap \textsf{{Sep}}$. 

For subsets $F,G$ of $\nn$, we write $F<G$ to mean that either $F=\varnothing$, $G=\varnothing$, or $\max F<\min G$. For $n\in\nn$ and $F\subset \nn$, we write $n\leqslant F$ to mean that $F\neq \varnothing$ and $n\leqslant \min F$.   For $f\in c_{00}$, we let $\supp(f)=\{i\in\nn:f(e_i)\neq 0\}$. For $f,g\in c_{00}$, we write $f<g$ to mean that $\supp(f)<\supp(g)$.  For $n\in\nn$ and $f\in c_{00}$, we write $n\leqslant f$ to mean that $n\leqslant \supp(f)$. 

Before introducing our spaces, we need to recall the construction of the \emph{$q$-convexification} $X^q$ of a Banach space $X$. First, if $X$ is a Banach space with Schauder basis $(e_j)_{j=1}^\infty$ and $C\ge 1$, then $(e_j)_{j=1}^\infty$ is said to be  \emph{$C$-unconditional},  if for all $(a_j)_{j=1}^\infty \in c_{00}$ and all $(\eps_j)_{j=1}^\infty \in \{-1,1\}^\Ndb$, 
$$\Bigl\|\sum_{j=1}^\infty \eps_ja_je_j\Bigr\|_X \leqslant C\Bigl\|\sum_{j=1}^\infty a_je_j\Bigr\|_X.$$
Let now $q\in [1,\infty)$ and $X$ be a Banach space with a normalized $1$-unconditional basis $(e_j)_{j=1}^\infty$. We set 
$$X^q=\Big\{x=(x_j)_{j=1}^\infty \in \Kdb^\Ndb,\ x^q=\sum_{j=1}^\infty |x_j|^qe_j \in X\Big\}$$
and endow it with the norm $\|x\|_{X^q}=\|x^q\|_X^{1/q}$. We also denote $(e_j)_{j=1}^\infty$ the sequence of coordinate vectors in $X^q$. It is clear that $(e_j)_{j=1}^\infty$ is a  normalized $1$-unconditional basis of $X^q$ and that $X^1$ is isometric to $X$. Also, the triangle inequality implies that $X^q$ is $q$-convex with constant $1$, meaning that for any $x_1,\ldots,x_n \in X^q$ (we write $x_k=(x_{k,j})_{j=1}^\infty$, for $1\le k\le n$), we have
$$\Big\|\sum_{j=1}^\infty\big(|x_{1,j}|^q+\cdots+|x_{n,j}|^q\big)^{1/q}e_j\Big\|_{X^q} \leqslant \big(\|x_1\|_{X^q}^q+\cdots+\|x_n\|_{X^q}^q\big)^{1/q}.$$
Note that that if $x_1,\ldots,x_n \in X^q$ have disjoint supports with respect to $(e_j)_{j=1}^\infty$, then $(x_1+\cdots+x_n)^q=x_1^q+\cdots+x_n^q$ and 
$$\|x_1+\cdots+x_n\|_{X^q}^q\leqslant \|x_1\|_{X^q}^q+\cdots+\|x_n\|_{X^q}^q.$$

We now proceed with the construction of our model spaces, starting with the model spaces for property $\textsf{A}_p$. Fix $p\in (1,\infty]$ and let $q$ be its conjugate exponent. Fix also $\theta \in (0,1)$. We recall that the Tsirelson space $T_{\theta^q}$ is the completion of $c_{00}$ under the implicitly defined norm 
\[\|x\|_{T_{\theta^q}} = \max\Bigl\{\|x\|_{c_0}, \theta^q \sup\bigl\{\sum_{i=1}^n \|I_ix\|_{T_{\theta^q}} : n\in \nn, n\leqslant I_1<\ldots < I_n\bigr\}\Bigr\}.\]
This norm is built as the limit of the inductively defined following sequence of norms: $\|x\|_0=\|x\|_\infty$ and, for $l \in \Ndb$:
$$\|x\|_l=\max\Bigl\{\|x\|_{l-1}, \theta^q \sup\bigl\{\sum_{i=1}^n \|I_ix\|_{l-1} : n\in \nn, n\leqslant I_1<\ldots < I_n\bigr\}\Bigr\}.$$
We refer the reader to the book by Casazza and Shura \cite{CasazzaShura}  for all the necessary background on Tsirelson spaces. We recall that the canonical basis of $T_{\theta^q}$ is $1$-unconditional and we let $T_{q,\theta}$ be the $q$-convexification of $T_{\theta^q}$. It follows from our preceding remarks that the canonical basis of $T_{q,\theta}$ is shrinking. It is also known that $T_{\theta^q}$ is reflexive, from which we have that its canonical basis is boundedly complete and therefore, so is the canonical basis of $T_{q,\theta}$. In particular $T_{q,\theta}$ is reflexive. 

It will be convenient for us to describe $T_{q,\theta}$ through a norming subset of its dual. So denote $(e_i^*)_{i=1}^\infty$ the dual basis of the canonical basis of $c_{00}$ and define 
\[K_0= \{\lambda e^*_i: i\in\nn, |\lambda|\leqslant 1\},\] \[K_l=K_{l-1}\cup \Bigl\{\theta \sum_{i=1}^n a_if_i: n\in\nn, n\leqslant f_1<\ldots < f_n, f_i\in K_{l-1}, f_1\neq 0, (a_i)_{i=1}^n \in B_{\ell_p^n}\Bigr\},\]  and \[K=\bigcup_{l=0}^\infty K_l.\]  When necessary, we will write $K^\theta_i$ in place of $K_i$.  

\begin{proposition} For all $x\in c_{00}$, $\|x\|_{T_{q,\theta}}=\sup_{f\in K}|f(x)|$. Moreover, the closed unit ball of $T_{q,\theta}^*$ is $\overline{\text{co}}(K)$, the closed convex hull of $K$.
\end{proposition}
\begin{proof} An easy induction shows that for all $x\in c_{00}$ and all $l\in \{0\}\cup \Ndb$, we have $\|x^q\|_{l}=\sup_{f\in K_l}|f(x)|^q$. We deduce immediately the first statement. The second assertion then follows from the reflexivity of $T_{q,\theta}$.
\end{proof}

\begin{proposition} Let $f_1,\ldots,f_n \in B_{T_{q,\theta}^*}$ such that $n\leqslant f_1<\cdots<f_n$. Then, 
$$\forall (b_i)_{i=1}^n\in B_{\ell_p^n},\ \ \theta\sum_{i=1}^nb_if_i \in B_{T_{q,\theta}^*}.$$
\end{proposition}

\begin{proof} Since $K$ is closed under interval projections, for positive integers $l\leqslant m$, 
$$\{f\in B_{T_{q,\theta}^*}, \supp(f)\subset [l,m]\}=\overline{\text{co}}\{g\in K,\ \supp(g)\subset [l,m]\}.$$
Let now $n\leqslant I_1<\cdots < I_n$ be intervals such that $\supp(f_i)\subset I_i$. It is clear from the definition of $K$ that for $g_1,\ldots,g_n$ so that $g_i\in K\cap \text{span}\{e_j: j\in I_i\}$, $\theta\sum_{i=1}^n b_ig_i\in K$. It then follows that $\theta\sum_{i=1}^nb_if_i \in \overline{\text{co}}(K)= B_{T_{q,\theta}^*}$.
\end{proof}

We now turn to our model spaces for property $\textsf{N}_p$ and modify the construction of the above norming subset. We let $U_{q,\theta}$ be the completion of $c_{00}$ with respect to the norm $\|x\|_{U_{q,\theta}}=\sup_{f\in L}|f(x)|$, where \[L_0= \{\lambda e^*_i: i\in\nn, |\lambda|\leqslant 1\},\] \[L_l=L_{l-1}\cup \Bigl\{\frac{\theta}{n^{1/p}} \sum_{i=1}^n a_if_i: n\in\nn, 2 \leqslant n\leqslant f_1<\ldots < f_n, f_i\in L_{l-1}, f_1\neq 0\Bigr\},\]  and \[L=\bigcup_{l=0}^\infty L_l.\] 
We also note that the norm $\|\cdot\|_{U_{q,\theta}}$ can be defined on $c_{00}$ by \[\|x\|_{U_{q,\theta}}=\lim_l |x|_{U_{q,\theta},l}\] where \[|x|_{U_{q,\theta},0}=\|x\|_{c_0}\] and \[|x|_{U_{q,\theta},l} = \max\Bigl\{|x|_{U_{q,\theta},l-1},  \sup\bigl\{{\theta}{n^{-1/p}}\sum_{i=1}^n |I_ix|_{U_{q,\theta},l-1}:2\leqslant n\leqslant I_1<\ldots I_n\bigr\}\Bigr\}.\]

With a similar argument we obtain

\begin{proposition} Let $f_1,\ldots,f_n \in B_{U_{q,\theta}^*}$ such that $n\leqslant f_1<\cdots<f_n$. Then, $$\theta n^{-1/p}\sum_{i=1}^nf_i \in B_{U_{q,\theta}^*}.$$
\end{proposition}

We shall need one last generalization of the above families. For that purpose, fix also $M=(m_n)_{n=1}^\infty \in [\Ndb]^\omega$ (the set of increasing sequences in $\Ndb$). We define the spaces $T_{M,q,\theta}$ and $U_{M,q,\theta}$ to be the completions of $c_{00}$ with respect to the norms $\|x\|_{T_{M,q,\theta}}=\sup_{f\in K_M}|f(x)|$ and $\|u\|_{U_{M,q,\theta}}=\sup_{f\in L_M}|f(x)|$, where \[K_{0,M}= \{\lambda e^*_i: i\in\nn, |\lambda|\leqslant 1\},\] \[K_{l,M}=K_{l-1,M}\cup \Bigl\{\theta \sum_{i=1}^n a_if_i: n\in\nn, m_n\leqslant f_1<\ldots < f_n, f_i\in K_{l-1,M}, f_1\neq 0, (a_i)_{i=1}^n \in B_{\ell_p^n}\Bigr\},\]  and \[K_M=\bigcup_{l=0}^\infty K_{l,M},\] and \[L_{0,M}= \{\lambda e^*_i: i\in\nn, |\lambda|\leqslant 1\},\] \[L_{l,M}=L_{l-1,M}\cup \Bigl\{\frac{\theta}{n^{1/p}} \sum_{i=1}^n a_if_i: n\in\nn, 2 \leqslant n, m_n\leqslant f_1<\ldots < f_n, f_i\in L_{l-1,M}, f_1\neq 0\Bigr\},\]  and \[L_M=\bigcup_{l=0}^\infty L_{l,M}.\] Of course, $T_{q,\theta}=T_{\nn,q,\theta}$ and $U_{q,\theta}=U_{\nn,q,\theta}$. 

Then, we have the following straightforward generalization of the previous propositions.

\begin{proposition} Let $f_1,\ldots,f_n \in B_{T_{M,q,\theta}^*}$ such that $m_n\leqslant f_1<\cdots<f_n$. Then, 
$$\forall (b_i)_{i=1}^n\in B_{\ell_p^n},\ \ \theta\sum_{i=1}^nb_if_i \in B_{T_{M,q,\theta}^*}.$$

Let $f_1,\ldots,f_n \in B_{U_{M,q,\theta}^*}$ such that $m_n\le f_1<\cdots<f_n$. Then, $$\theta n^{-1/p}\sum_{i=1}^nf_i \in B_{U_{M,q,\theta}^*}.$$

\end{proposition}

We are now ready to relate these families to properties $\textsf{A}_p$ and $\textsf{N}_p$.
\begin{proposition}\label{modelApNp} Fix $1<p\leqslant \infty$ and let $1/p+1/q=1$. For any $\theta\in (0,1)$ and $M\in[\nn]^\omega$, $T^*_{M,q,\theta}$ has $\textsf{\emph{A}}_p$ and $U^*_{M,q,\theta}$ has $\textsf{\emph{N}}_p$. 

\end{proposition}

\begin{proof} For $c>\theta^{-1}$, let us quickly describe the winning strategy for Player I in the spatial $A(c,p,n)$ game in $T^*_{M,q,\theta}$. Player I chooses $Y_1=\{x \in T^*_{M,q,\theta},\ m_n\le \supp(x)\}$. Then, after each choice of $x_i$ by Player II, Player I picks $Y_{i+1}=\{x \in T^*_{M,q,\theta},\ k_{i+1}\le \supp(x)\}$, for some $k_{i+1}>k_i$ so that, there exist small enough perturbations $y_1,\ldots,y_n$ of $x_1,\ldots,x_n$ satisfying $m_n\leqslant y_1<\cdots<y_n$ to ensure, thanks to the previous proposition, that for all $(a_i)_{i=1}^n \in B_{\ell_p^n}$,
$$\Big\|\sum_{i=1}^n a_iy_i\Big\|_{T_{M,q,\theta}^*}^p \leqslant \theta^{-p}\ \ \text{and}\ \ \Big\|\sum_{i=1}^n a_ix_i\Big\|_{T_{M,q,\theta}^*}^p \leqslant c.$$

To address property $\textsf{N}_p$ in $U^*_{M,q,\theta}$, we similarly use the fact that for $n\in \Ndb$ and $m_n\leqslant f_1<\cdots<f_n$, with $f_i\in B_{U^*_{M,q,\theta}}$, we have
${\theta}{n^{-1/p}}\sum_{i=1}^n f_i\in B_{U^*_{M,q,\theta}}.$
\end{proof}

The next result goes somewhat in the other direction. Any $\textsf{{A}}_p$ (resp. $\textsf{{N}}_p$) space has a $T_{q,\theta}^*$ (resp. $U_{q,\theta}^*$) like behavior. 
\begin{theorem}\label{upper1} 
Let $X$ be a Banach space and $p\in (1,\infty]$. Let $q$ be the conjugate exponent of $p$. 
\begin{enumerate}[(i)]
\item If $X$ has $\textsf{\emph{A}}_p$, then there exists $\theta_0\in (0,1)$ such that for all $\theta \in (0,\theta_0]$ Player I has a winning strategy in the  spatial $(T_{q,\theta}^*,1)$ game in $X$. 
\item  If $X$ has $\textsf{\emph{N}}_p$, then there exists $\theta_0\in (0,1)$ such that  for all $\theta \in (0,\theta_0]$ Player I has a winning strategy in the  spatial $(U_{q,\theta}^*,1)$ game in $X$. 
\end{enumerate}
\end{theorem}

\begin{proof}$(i)$ There exists a constant $c>1$ such that for all $n\in\nn$, Player I has a winning strategy $\chi_n$ in the compact spatial $A(c,p,n)$ game.  We define a winning  spatial $\omega$ strategy $\chi$ for Player I in the $(T_{q,\theta}^*,1)$ game, for  $\theta \in (0,\frac{1}{c}]$.  

Let $\chi(\varnothing)=\chi_1(\varnothing)$.   Assume that for some $l\in\nn$,  $\chi((x_i)_{i=1}^k)$ has been defined for all $(x_i)_{i=1}^k\in B_X^{<l}$.   Fix $(x_i)_{i=1}^l\in {B_X}^l$.    Define \[\chi((x_i)_{i=1}^l) = \Bigl(\bigcap_{j=1}^{l+1} \chi_j(\varnothing)\Bigr)\cap \Bigl(\bigcap_{j=1}^{l+1} \bigcap_{k=1}^{j-1} \bigcap_{I_1<\ldots<I_k,I_i\subset [1,l]} \chi_j\Bigl(\bigl(B_X\cap \text{span}\{x_m:m\in I_i\}\bigr)_{i=1}^k\Bigr)\Bigr). \]

This completes the recursive construction. We now fix $(x_i)_{i=1}^\infty$ $\chi$-admissible. Assume now that $2\le n\in\nn$ and $n\leqslant I_1<\ldots < I_n$ and denote $C_i=B_X\cap \text{span}\{x_m:m\in I_i\}$. We claim that $(C_i)_{i=1}^n$ is $\chi_n$-admissible. Indeed, for any $1\le j\le n$ and any $m\in I_j$, we have that $m\ge n$, which implies that 
$$x_m\in  \bigcap_{k=1}^{n-1} \bigcap_{J_1<\ldots<J_k,J_i\subset [1,m-1]} \chi_n\Bigl(\bigl(B_X\cap \text{span}\{x_m:m\in J_i\}\bigr)_{i=1}^k\Bigr)\Bigr.$$
In particular, $x_m \in \chi_n(C_1,\ldots,C_{j-1})$, which proves our claim. Since $\chi_n$ is a winning strategy for Player I in the compact spatial $A(c,p,n)$ game, we obtain that
\[\forall (u_i)_{i=1}^n \in \prod_{i=1}^n C_i\ \ \forall (a_i)_{i=1}^n\in B_{\ell_p^n},\ \ \theta\sum_{i=1}^n a_iu_i\in B_X\]  
For $a=(a_i)_{i=1}^\infty \in c_{00}$, we define now $A(\sum_{i=1}^\infty a_ie_i^*)=\sum_{i=1}^\infty a_ix_i$, where $(e_i^*)_{i=1}^\infty$ is the canonical basis of $T^*_{q,\theta}$. It remains to show that $A$ maps $B_{T^*_{q,\theta}}\cap c_{00}$ into $B_X$. We adopt the notation used in the construction of $T_{q,\theta}$ for the sets $K=\cup_{l=0}^\infty K_l$ and recall that $B_{T^*_{q,\theta}}$ is the closed, convex hull of $K$. Therefore, it is sufficient to show that $A(K)\subset B_X$. For this we prove by induction on $l$ that $A(K_l)\subset B_X$. The base case follows from the fact that any admissible sequence must lie in $B_X$. Assume the result has been proved for $l\ge 0$, and let $f \in K_l\setminus K_{l-1}$. Then there exists $n\geqslant 2$, $n\leqslant f_1<\cdots<f_n \in K_{l-1}$ and $(a_i)_{i=1}^n\in B_{\ell_p^n}$ such that $f=\theta \sum_{i=1}^na_if_i$. By induction hypothesis, we have that $u_i=A(f_i)\in B_X$. Note also that $2\leqslant n\le u_1<\cdots<u_n$. The above discussion then implies that $Af=\theta\sum_{i=1}^n a_iu_i \in B_X$.

The proof of $(ii)$ is an inessential modification of $(i)$. 
\end{proof}

Let $T$ be a Banach space with basis $(e_n)_{n=1}^\infty$ and $C\ge 1$. We recall that that $(e_n)_{n=1}^\infty$ is \emph{$C$-right dominant} if for all $(a_k)_{k=1}^n \in \Kdb^n$, $i_1<\cdots<i_n$ and $j_1<\cdots<j_n$ with $i_k\le j_k$ for all $1\le k \le n$, we have
$$\Big\|\sum_{k=1}a_ke_{{i_k}}\Big\|_T \leqslant C \Big\|\sum_{k=1}a_ke_{{j_k}}\Big\|_T.$$
The definition of a $C$-left dominant basis is obtained by exchanging the places of $i_k$'s and $j_k$'s in the above inequality. 

We will need the following lemma on interlaced subsequences of the canonical basis $(e_n)_{n=1}^\infty$ of our model spaces. 

\begin{lemma}\label{shuffle} Let $p\in (1,\infty]$, $q$ be its conjugate exponent and $\theta\in (0,1)$. 
Let $(k_i)_{i=1}^\infty, (l_i)_{i=1}^\infty$ be two sequences of integers such that $1\le k_1< l_1<k_2< l_2<\ldots$.   Then, for all $(a_i)_{i=1}^\infty \in c_{00}$
\begin{enumerate}[(i)]
\item  $\big\|\sum_{i=1}^\infty a_ie_{l_i}\big\|_{T_{q,\theta}} \leqslant 3 \big\|\sum_{i=1}^\infty a_ie_{k_i}\big\|_{T_{q,\theta}}$.
\item  $\big\|\sum_{i=1}^\infty a_ie_{l_i}\big\|_{U_{q,\frac{\theta}{2}}} \leqslant  \big\|\sum_{i=1}^\infty a_ie_{k_i}\big\|_{U_{q,\theta}}$.
\end{enumerate}
\end{lemma}

\begin{proof}$(i)$ This follows from the fact that $T_{q,\theta}$ is the $q$-convexification of $T_{\theta^q}$, $1$-right dominance of the canonical basis of $T_{\theta^q}$, and the fact that for any $1\leqslant n_1<n_2<\cdots,$ $(e_{n_{2i}})_{i=1}^\infty \lesssim_c  (e_{n_{i}})_{i=1}^\infty$ in $T_{\theta^q}$ (see Proposition I.12 in  \cite{CasazzaShura}). 

$(ii)$ We prove by induction on $r$ that for any $(a_i)_{i=1}^\infty \in c_{00}$, \[\Bigl|\sum_{i=1}^\infty a_ie_{l_i}\Bigr|_{\frac{\theta}{2},r}\leqslant \Bigl|\sum_{i=1}^\infty a_ie_{k_i}\Bigr|_{\theta,r}.\]   The $r=0$ case is trivial.   Assume the result holds for some $r$ and fix $(a_i)_{i=1}^\infty\in c_{00}$. Let $x=\sum_{i=1}^\infty a_ie_{l_i}$ and $y=\sum_{i=1}^\infty a_ie_{k_i}$. If $|x|_{\frac{\theta}{2},r+1}=|x|_{\frac{\theta}{2},r}$, then  
\[ |x|_{\frac{\theta}{2},r+1}=|x|_{\frac{\theta}{2},r} \leqslant |y|_{\theta,r}\leqslant |y|_{\theta,r+1}.\]  
Assume $|x|_{\frac{\theta}{2},r+1}>|x|_{\frac{\theta}{2},r}$. Then for some $2\leqslant n<I_1<\ldots <I_n$, 
\[|x|_{\frac{\theta}{2},r+1}=\frac{\theta}{2n^{1/p}}\sum_{i=1}^n |I_ix|_{\frac{\theta}{2},r}.\]  
Note that for each $1\leqslant i\leqslant n$, $I_ix\neq 0$. Indeed, note first that $I_ix$ must be non-zero for at least two values of $i$, otherwise it would be true that $|x|_{\frac{\theta}{2},r+1}=|x|_{\frac{\theta}{2},r}$. Then, if $I_ix=0$ for some $i$, then we could omit all such $i$'s and replace $\frac{\theta}{2n^{1/p}}$ with $\frac{\theta}{2m^{1/p}}$, where $m=|\{i:I_ix\neq 0\}|\in (1,n)$, which would lead to the following contradiction  \[|x|_{\frac{\theta}{2},r+1}\geqslant \frac{\theta}{2m^{1/p}}\sum_{i:I_ix\neq 0}|I_ix|_{\frac{\theta}{2},r}>\frac{\theta}{2n^{1/p}}\sum_{i=1}^n|I_ix|_{\frac{\theta}{2},r}=|x|_{\frac{\theta}{2},r+1}.\]  
So $\{j\in\nn: l_j\in I_1\}\neq \emptyset$ and we can set $t=\min \{j\in\nn: l_j\in I_1\}$. Define 
\[G_i=\left\{\begin{array}{ll} \{j\in\nn: l_j\in I_i\} & : 1<i\leqslant n \\ \{j\in\nn:l_j\in I_1\}\setminus \{t\} & : i=1, \end{array}\right.\]  
\[H_i =\left\{\begin{array}{ll} \{j\in\nn: l_j\in I_i\} & : 1<i\leqslant n \\ \{j\in\nn:k_j\in I_1\} & : i=1, \end{array}\right.\] and let $J_1, \ldots, J_n$ be the smallest intervals such that $\{k_j:j\in H_i\}\subset J_i$. Note that $n\leqslant J_1<\ldots <J_n$. By the properties of $(k_i)_{i=1}^\infty, (l_i)_{i=1}^\infty$, we have 
\begin{align*} 
|x|_{\frac{\theta}{2},r} & = \frac{\theta}{2n^{1/p}}\sum_{i=1}^n |I_ix|_{\frac{\theta}{2},r} \leqslant \frac{\theta}{2n^{1/p}}|a_t||e_{l_t}|_{\frac{\theta}{2},r} + \frac{\theta}{2n^{1/p}} \sum_{i=1}^n \Bigl|\sum_{j\in G_i}a_je_{l_j}\Bigr|_{\frac{\theta}{2},r} \\ 
& \leqslant \frac{1}{2}\Bigl|\sum_{i=1}^\infty a_ie_{k_i}\Bigr|_{\theta,0}+ \frac{\theta}{2n^{1/p}} \sum_{i=1}^n \Bigl|\sum_{j\in H_i}a_je_{k_j}\Bigr|_{\theta,r} \ \ \ \text{(by the inductive hypothesis)}
\\ & = \frac{1}{2}\Bigl|\sum_{i=1}^\infty a_ie_{k_i}\Bigr|_{\theta,0}+ \frac{\theta}{2n^{1/p}} \sum_{i=1}^n \Bigl|J_i\sum_{j=1}^\infty a_je_{k_j}\Bigr|_{\theta,r} \leqslant \frac{1}{2}\Bigl|\sum_{i=1}^\infty a_ie_{k_i}\Bigr|_{\theta,0}+\frac{1}{2}\Bigl|\sum_{i=1}^\infty a_ie_{k_i}\Bigr|_{\theta,r+1}\\ 
&\leqslant\Bigl|\sum_{i=1}^\infty a_ie_{k_i}\Bigr|_{\theta,r+1}=|y|_{\theta,r+1}. 
\end{align*}

\begin{remark}\upshape  If $k_1<l_1<k_2<l_2<\ldots$, then also $l_1<k_2<l_2<k_3<\ldots$, $k_2<l_2<k_3<l_3<\ldots$, etc.    From this it follows that for $m=0,1,2,\ldots$,  
\[T_{q,\theta} \supset (e_{l_{m+i}})_{i=1}^\infty\lesssim_{3^{2m+1}} (e_{k_i})_{i=1}^\infty\subset T_{q,\theta}\]  
and 
\[U_{q,2^{-(2m+1)}\theta} \supset (e_{l_{m+i}})_{i=1}^\infty\lesssim (e_{k_i})_{i=1}^\infty\subset U_{q,\theta}.\]
\end{remark}

\end{proof}

\subsection{The spaces $Z^T_\wedge(\textsf{E})$ and $Z^T_\vee(\textsf{E})$}\label{pressdown}

We recall that a \emph{finite dimensional decomposition} for a Banach space $Z$ is a sequence $\textsf{E}=(E_n)_{n=1}^\infty$ of finite dimensional, non-zero subspaces of $Z$ such that for any $z \in Z$, there exists a unique sequence $(z_n)_{n=1}^\infty \in \prod_{n=1}^\infty E_n$ such that $z=\sum_{n=1}^\infty z_n$. Then, we let $P^\textsf{E}_n$ denote the canonical projections $P^\textsf{E}_n(z) = z_n$, where $z=\sum_{n=1}^\infty z_n$ and $(z_n)_{n=1}^\infty \in \prod_{n=1}^\infty E_n$. For a finite or cofinite subset $I$ of $\nn$, we let $P^\textsf{E}_I=I^\textsf{E}=\sum_{n\in I}P^\textsf{E}_n$.  When no confusion can arise, we omit the superscript and simply denote $I^\textsf{E}$ by $I$. It follows from the principle of uniform boundedness that $\sup\{\|I^\textsf{E}\|,\ I\subset \Ndb\ \text{is an interval}\}$ is finite. We refer to this quantity as the \emph{projection constant} of $\textsf{E}$  in $Z$. If the projection constant of $\textsf{E}$  in $Z$ is $1$, we say $\textsf{E}$ is \emph{bimonotone}. It is well-known that if $\textsf{E}$ is an FDD for $Z$, then there exists an equivalent norm $|\ |$ on $Z$ such that $\textsf{E}$ is a bimonotone FDD of $(Z,|\ |)$. We denote $c_{00}(\textsf{E})$ the space of finite linear combinations of elements in $E_1,\ldots,E_n,\ldots$.

A sequence $\textsf{F}=(F_n)_{n=1}^\infty$ is called a \emph{blocking} of the the FDD $\textsf{E}$ of $Z$ if there exists an increasing sequence $1=m_0<m_1<\cdots<m_n<\cdots$ such that for all $n\in \Ndb$, $F_n=\oplus_{i=m_{n-1}}^{m_n-1}E_i$.
	
We also need to recall some basics on dual FDD's. If $Z$ is a Banach space with FDD $\textsf{E}=(E_n)_{n=1}^\infty$, we let $\textsf{E}^*$ denote the sequence $(E^*_n)_{n=1}^\infty$. Here, $E^*_n$ is identified with the sequence $((P^\textsf{E}_n)^*(Z^*))_{n=1}^\infty$. This identification need not be isometric if $\textsf{E}$ is not bimonotone in $Z$.  We let $Z^{(*)}=\overline{c_{00}(\textsf{E}^*)} \subset Z^*$. The FDD $\textsf{E}$ is said to be \emph{shrinking} if $Z^{(*)}=Z^*$, which occurs if and only if any bounded block sequence with respect to $\textsf{E}$ is weakly null. The FDD $\textsf{E}$ is said to be \emph{boundedly complete} if $\textsf{E}^*$ is a shrinking FDD of $Z^{(*)}$ (in that case $Z$ is canonically isomorphic to $(Z^{(*)})^*)$.

%We let $[\nn]^{<\omega}$ denote the set of all finite, non-empty subsets of $\nn$.  We let $[\nn]$ denote the set of all infinite subsets of $\nn$.  We identify subsets of $\nn$ with strictly increasing sequences of members of $\nn$ by identifying a set with the sequence obtained by listing its members in strictly increasing order.  

Given a Banach space $Z$ with FDD $\textsf{E}$ and a Banach space $T$ with normalized, $1$-unconditional basis, we define two associated spaces, $Z^T_\wedge(\textsf{E})$ and $Z^T_\vee(\textsf{E})$.    Each will be the completion of $c_{00}(\textsf{E})$ with respect to the quantities $\|\cdot\|_\wedge, \|\cdot\|_\vee$ defined below.    

For $z \in c_{00}(\textsf{E})$, we define 
\[\|z\|_\vee = \sup \Bigl\{\Bigl\|\sum_{i=1}^\infty \|I^\textsf{E}_i z\|_Z e_{\min I_i}\Bigr\|_T: I_1<I_2<\ldots, I_i\text{\ an interval}\Bigr\}.\]
We call this norm the \emph{lift up norm} associated with $Z$, $\textsf{E}$ and $T$.

We also define \[[z]_\wedge = \inf\Bigl\{\Bigl\|\sum_{i=1}^\infty \|I^\textsf{E}_iz\|_Z e_{\min I_i}\Bigr\|_T: I_1<I_2<\ldots, \nn=\cup_{i=1}^\infty I_i\Bigr\}\] and \[\|z\|_\wedge = \inf\Bigl\{\sum_{i=1}^n [z_i]_\wedge:n\in\nn, z_i\in c_{00}(\textsf{E}), z=\sum_{i=1}^n z_i\Bigr\}.\] 
We call this norm the \emph{press norm} associated with $Z$, $\textsf{E}$ and $T$. 

It is easily checked that $\textsf{E}$ is a FDD for $Z^T_\wedge(\textsf{E})$ and $Z^T_\vee(\textsf{E})$. The following classical convexity lemma will be useful. 
\begin{lemma}\label{pressdown_conv} Assume moreover that $\textsf{\emph{E}}$ is a bimonotone FDD of $Z$. Let $I_1<\cdots<I_n$ intervals of $\Ndb$ and $z_1,\ldots,z_n \in c_{00}(\textsf{\emph{E}})$ with $\supp(z_i)\subset I_i$, then  
$$z_1+\cdots+z_n \in \overline{\text{conv}}\big\{y_1+\cdots+y_n,\ \supp(y_i)\subset I_i,\ [y_i]_{\wedge}\le \|z_i\|_{\wedge}\big\}.$$
\end{lemma}

The duality between press down and lift up norms is described by the following proposition (see Proposition 2.1 in \cite{CauseyNavoyan}). 

\begin{proposition} Let $Z$ be a Banach space with bimonotone FDD $\textsf{\emph{E}}$ and let $T$ be a Banach space with normalized, $1$-unconditional basis. 

\begin{enumerate}[(i)]\item $(Z^T_\wedge(\textsf{\emph{E}}))^{(*)}=(Z^{(*)})_\vee^{T^{(*)}}(\textsf{\emph{E}}^*)$. 

\item $(Z^T_\vee(\textsf{\emph{E}}))^{(*)}=(Z^{(*)})_\wedge^{T^{(*)}}(\textsf{\emph{E}}^*)$. 

\end{enumerate}

\end{proposition}

It is known that if the basis of $T$ is shrinking, then the FDD $\textsf{E}$ of $Z^T_\wedge(\textsf{E})$ is shrinking. In item $(iii)$ of the next proposition, we include a separate proof which is illustrative in the case that $T$ has $\textsf{N}_p$. 

\begin{proposition} Fix $1<p\leqslant \infty$ and let $T$ be a Banach space with normalized, $1$-unconditional, shrinking basis. Let $Z$ be a Banach space with FDD $\textsf{\emph{E}}$.  
\begin{enumerate}[(i)]
\item  If $T$ has $\textsf{\emph{A}}_p$, then so does $Z^T_\wedge(\textsf{\emph{E}})$. 
\item If $T$ has $\textsf{\emph{N}}_p$, then so does $Z^T_\wedge(\textsf{\emph{E}})$. 
\item If $T$ has $\textsf{\emph{N}}_p$, then $\textsf{\emph{E}}$ is shrinking in $Z^T_\wedge(\textsf{\emph{E}})$. 
\end{enumerate}

\end{proposition}

\begin{proof} Let us explain the argument for $(i)$. First, we assume, as we may after renorming, that the FDD $\textsf{E}$ of $Z$ is monotone. Since $T$ has $\textsf{A}_p$, there exists $c>0$ such that for all $n\in \Ndb$, Player I has a winning strategy in the spatial $A(c,p,n)$ game on $T$. Let us now fix $n\in \Ndb$. Let $\chi:B_T^{<n}\to \co(T)$ be a winning spatial strategy for Player I. Since the basis $(e_i)$ of $T$ is shrinking, we may assume, by another approximation argument, that $\chi$ takes values in $\{T_n,\ n\in \Ndb\}$, where $T_n$ denotes the closed linear span of $\{e_k,\ k\geqslant n\}$. We may also assume that whenever $u_1,\ldots,u_n \in c_{00}$ are such that $(u_1,\ldots,u_n)$ is $\chi$-admissible, we have $u_1<\cdots<u_n$. We now try to define a winning strategy for Player I in the $A(c,p,n)$ game on $Z^T_\wedge(\textsf{{E}})$. For that purpose we denote $Z_n$ the closed linear span of $\cup_{k\ge n}E_k$ in $Z^T_\wedge(\textsf{{E}})$. Again by approximation and replacing $c$ by $c'>c$, it is enough to define $\psi$ on $(B_{Z^T_\wedge(\textsf{{E}})}\cap c_{00}(\textsf{E})^{<n}$. Then everything is in place to support our next claim. We can build $\psi$ with the property that if $z_1,\ldots,z_n \in c_{00}(\textsf{E})$ are such that $(z_1,\ldots,z_n)$ is $\psi$-admissible, then $z_1<\cdots<z_n$ and for each $1\leqslant i\leqslant n$, there exists intervals $I_{i,1}<\cdots<I_{i,j_i}$ covering the support of $z_i$ so that $[z_i]_{\wedge}= \|u_i\|_T$, where $u_i=\sum_{j=1}^{j_i} \|I^\textsf{E}_jz_i\|_Z e_{\min I_{i,j}}$ and $(v_1,\ldots,v_n)$ is $\chi$-admissible, where $v_i=u_i(\|u_i\|_T)^{-1}$. Then
$$\Big\|\sum_{i=1}^n a_iz_i\Big\|_{\wedge}^p \leqslant \Big[\sum_{i=1}^n a_iz_i\Big]_\wedge^p \leqslant \Big\|\sum_{i=1}^n a_iu_i\Big\|_{T}^p \leqslant c^p\sum_{i=1}^n |a_i|^p[z_i]_\wedge^p.$$
The second inequality is due to the fact that $(I_{i,j})_{(i,j)}$ is an interval covering of the support of $\sum_{i=1}^n a_iz_i$ and the last inequality comes from the $\chi$-admissibility of $(v_1,\ldots,v_n)$. Then, the conclusion of our proof follows from Lemma \ref{pressdown_conv}.

The proof of $(ii)$ is similar.

We prove $(iii)$. Assume $T$ and therefore $Z^T_\wedge(\textsf{{E}})$ have $\textsf{N}_p$. Then, $Z^T_\wedge(\textsf{{E}})$ has $\textsf{T}_s$, for $1<s<p$. It follows that any bounded block sequence $(z_k)_k$ in $Z^T_\wedge(\textsf{{E}})$ admits a subsequence which is dominated by the $\ell_r$ basis, where $r$ is the conjugate exponent of $s$. We then easily deduce that $(z_k)_k$ is weakly null, which shows that $\textsf{E}$ is a shrinking FDD of $Z^T_\wedge(\textsf{{E}})$.

\end{proof}

As a direct consequence of the previous corollary and Proposition \ref{modelApNp}, we get
\begin{corollary} Let $p\in (1,\infty]$, $q$ be the conjugate exponent of $p$, $\theta \in (0,1)$ and $M\in[\nn]^\omega$. Then, for any Banach space $Z$ with  FDD $\textsf{\emph{E}}$, $Z^{T^*_{q,\theta,M}}_\wedge(\textsf{\emph{E}})$ has $\textsf{\emph{A}}_p$ and $Z^{U^*_{q,\theta,M}}_\wedge(\textsf{\emph{E}})$ has $\textsf{\emph{N}}_p$. 
\end{corollary}

\subsection{Gliding hump and quotient maps}\label{gliding}
In this subsection we gather a few general results on quotient maps with an application to the general $(T,c)$ games.  

\begin{proposition}\label{dual1} 
Let $X$ be a Banach space and $Z$ be a Banach space with shrinking FDD $\textsf{\emph{E}}$ and let $Q:Z\to X$ be a quotient map. For simplicity, we will identify in our notation $X^*$ with its image in $Z^*$ by the isometry $Q^*$.  Then, for any finite-codimensional subspace $Y$ of $X$,  any $\delta\in (0,\frac{1}{20})$, and any $j\in\nn$, there exists $l\in (j,\infty)$ such that for any interval $I\subset [l,\infty)$ and any $x^*\in S_{X^*}$ such that $\|x^*-P^{\textsf{\emph{E}}^*}_Ix^*\|\leqslant \delta$,  there exists $z\in B_Z$ such that $Qz\in Y$ and $\text{\emph{Re\,}}x^*(Qz)\geqslant \frac13$. 
\end{proposition}

\begin{proof} If the result were not true, then for some $j\in\nn$ and every $l>j$, there would exist an interval $I_l\subset [l,\infty)$ and $x^*_l\in S_{X^*}$ such that $\|x^*_l-P^{\textsf{E}^*}_{I_l}x^*_l\|\leqslant \delta$ and for each $z\in B_Z\cap Q^{-1}(Y)$, $|x^*_l(Qz)|<\frac13$. For each $l>j$, we can choose $x_l\in (1-\delta)B_X$ such that $\text{Re\ }x^*_l(x_l)>1-2\delta$. Since $Y$ is finite codimensional in $X$, by passing to a subsequence and relabeling, we can assume, using compactness in a finite-dimensional complement of $Y$, that for all $l_1, l_2>j$, $x_{l_2}-x_{l_1} \in \delta B_X + 2 B_Y$. 

Fix any $l_1>j$ and $z_1\in B_Z$ such that $Qz_1=x_{l_1}$. Note that, since $\textsf{{E}}$ is shrinking,
\begin{align*} \underset{l}{\lim\sup} |x^*_l(x_{l_1})| &  = \underset{l}{\lim\sup} |x^*_l(Qz_1)| \leqslant \underset{l}{\lim\sup} \|x^*_l-P^{\textsf{E}^*}_{I_l}x^*_l\| + \underset{l}{\lim\sup} |P^{\textsf{E}^*}_{I_l}x^*_l(z_1)| \leqslant \delta. 
\end{align*} 
Therefore for sufficiently large $l_2>l_1$, $|x^*_{l_2}(x_{l_1})|<2\delta$.   Let $y=\frac{x_{l_2}-x_{l_1}}{2}\in \frac{\delta}{2} B_X+B_Y$ and fix $u\in \frac{\delta}{2}B_X$ and $v\in \delta B_Z$ such that $y-u\in B_Y$ and $Qv=u$.    Fix $z_2\in B_Z$ such that $Qz_2=x_{l_2}$ and  let 
\[z_0=\frac{z_2-z_1}{2} -v\in (1+\delta)B_Z.\] 
Note that $Qz_0=y-u\in B_Y$ and 
\begin{align*}
\text{Re\ }x^*_{l_2}(Qz_0) & \geqslant \frac12\text{Re\ }x^*_{l_2}(x_{l_2}) - \frac12|x^*_{l_2}(x_{l_1})| - |x^*_{l_2}(v)| \geqslant \frac{1-2\delta}{2}-\delta-\delta=\frac12-3\delta. 
\end{align*}  
Finally, for $z=\frac{z_0}{1+\delta}\in B_Z$, we have that $Qz\in Y$ and 
\[\text{Re\ }x^*_{l_2}(Qz) \geqslant \frac{\frac12-3\delta}{1+\delta}\geqslant \frac{\frac12-\frac{3}{20}}{1+\frac{1}{20}}= \frac13.\]  
This contradiction finishes the proof. 
\end{proof}

\begin{lemma}\label{net}
 Let $M$ be a metric space and let $f:M\to M$ be a function.  Suppose that $G\subset M$ is such that $f(G)\subset K$ for some compact subset $K$ of $M$.   Then for any $\delta>0$, the set $\{x\in G:d_M(x,f(x))\leqslant \delta\}$ admits a finite $4\delta$-net, whenever it is non empty. 
\end{lemma}

\begin{proof} So assume $H=\{x\in G:d_M(x,f(x))\le \delta\}\neq \varnothing$. Fix a finite $\delta$-net $F$ of $K$. Define $\eta:H\to F$ by letting $\eta(x)$ be such that $d_M(f(x),\eta(x))\le \delta$.  Let $\nu:\eta(H)\to H$ be such that for each $y\in \eta(H)$, $\nu(y)\in \eta^{-1}(\{y\})$.   Note that $\nu(\eta(H))\subset H$ is finite. We show that $\nu(\eta(H))$ is a $4\delta$-net for $H$.   Fix $x\in H$ and let $y=\eta(x)\in \eta(H)$ and $z=\nu(y)\in \nu(\eta(H))$. Then 
\begin{align*} 
d_M(x,z) & \leqslant d_M(x,f(x))+ d_M(f(x), y) + d_M(y,f(z)) + d_M(f(z),z) \\ 
& = d_M(x,f(x))+ d_M(f(x), \eta(x)) + d_M(\eta(z),f(z)) + d_M(f(z),z) \\ & \leqslant 4\delta. \end{align*} 
\end{proof}

For the next two results, we adopt some convenient notation. For positive integers $1\leqslant r_1<r_2<\ldots$ and any $(x_i)_{i=1}^\infty$, we let $(x_i, r_i)_{i=1}^\infty$ denote the sequence $(u_j)_{j=1}^\infty$ such that \[u_j = \left\{\begin{array}{ll} x_i & : j=r_i \\ 0 & : j\in \nn\setminus \{r_1, r_2, \ldots\}.\end{array}\right.\] That is, $(x_i, r_i)_{i=1}^\infty$ is the sequence which has $x_1$ in position $r_1$, $x_2$ in position $r_2$, $\ldots$, and all other positions are occupied by $0$. We make a similar definition for finite sequences $(x_i, r_i)_{i=1}^n$. 

\begin{proposition}\label{hawk} Let $X$ be a Banach space, let $Z$ be a Banach space with shrinking FDD $\textsf{\emph{E}}$, and let $Q:Z\to X$ be a quotient map. Let $\chi$ be a spatial $\omega$-strategy on $X$. For any strictly decreasing, null sequence  $(\delta_i)_{i=1}^\infty \subset (0,\frac{1}{20})$, there exists a blocking $\textsf{\emph{F}}$ of $\textsf{\emph{E}}$ such that the following holds:  For any integers $1\leqslant r_1<r_2<\ldots$, any intervals $I_i$ such that $I_i\subset (r_i, r_{i+1})$, and any $(x^*_i)_{i=1}^\infty\subset S_{X^*}$ such that $\|x^*_i-P^{\textsf{F}^*}_{I_i}x^*_i\|\leqslant \delta_i$ for all $i\in\nn$, there exist $(y^*_i)_{i=1}^\infty\subset S_{X^*}$ and $(z_i)_{i=1}^\infty\subset B_Z$ such that 
\begin{enumerate}[(i)]
\item $(Qz_i, r_i)_{i=1}^\infty$ is $\chi$-admissible, 
\item $\|x^*_i-y^*_i\|\leqslant 4\delta_i$ for all $i\in\nn$, 
\item $\text{\emph{Re}\ }y^*_i(Qz_i)\geqslant \frac13$ for all $i\in\nn$, 
\item $|y^*_i(Qz_j)|\leqslant 2\delta_{\max\{i,j\}}$ for all $i,j\in\nn$ with $i\neq j$. 
\end{enumerate}
\end{proposition}

Before the proof, we introduce some notation. For sets $\Lambda, \Upsilon$ and a subset $B$ of $\Lambda^{<\omega}$ which is closed under taking non-empty initial segments, we say a function $\Sigma:B\to \Upsilon^{<\omega}$ is \emph{monotone} if whenever $\sigma$ is a non-empty initial segment of $\tau\in B$, then $\Sigma(\sigma)$ is an initial segment of $\Sigma(\tau)$.   If $\Sigma:B\to \Upsilon^{<\omega}$ is monotone and length-preserving, then if $\Sigma((\lambda_i)_{i=1}^n)=(\upsilon_i)_{i=1}^n$, $\Sigma((\lambda_i)_{i=1}^m)=(\upsilon_i)_{i=1}^m$ for all $1\leqslant m\leqslant n$.    The \emph{body} of $B$ is given by \[[B]= \{(\lambda_i)_{i=1}^\infty\in \Lambda^\omega:(\forall n\in\nn)((\lambda_i)_{i=1}^n\in B)\}.\]     If $\Sigma:B\to \Upsilon^{<\omega}$ is monotone and length-preserving, then there is a natural extension $\overline{\Sigma}:[B]\to \Upsilon^\omega$ given by $\overline{\Sigma}((\lambda_i)_{i=1}^\infty)$ is the infinite sequence whose initial segments are given by $\Sigma((\lambda_i)_{i=1}^n)$, $n\in\nn$. 

\begin{proof} In the proof, for a sequence $\sigma$, we let $\text{im}(\sigma)$ denote the set of members of $\sigma$.  For a sequence $\sigma=(I_i, x_i)_{i=1}^n$ of pairs, we let $\text{im}_2(\sigma)=\{x_1, \ldots, x_n\}$, the set of all second members of these pairs. 

So let us fix a spatial strategy $\chi: B_X^{<\omega}\to \co(X)$ and $(\delta_i)_{i=1}^\infty \subset (0,\frac{1}{20})$ strictly decreasing to $0$. For each non empty finite interval $I$ and $i\in\nn$, let \[G_{I,i}=\{x^*\in S_{X^*}:\|x^*-P^{\textsf{E}^*}_Ix^*\|\leqslant \delta_i\}.\]  
Clearly, for each $I$ and $i$ $G_{I,i}$ is non empty. So, by Lemma \ref{net} there exists a finite $4\delta_i$-net $N_{I,i}$ of $G_{I,i}$. Then we define 
\[B=\Bigl\{(I_i, x^*_i)_{i=1}^n:n\in\nn, I_1<\ldots <I_n, I_i \text{ intervals},(x^*_i)_{i=1}^n \in \prod_{i=1}^n N_{I_i,i}\Bigr\},\] 
which is a subset of $(2^\nn\times S_{X^*})^{<\omega}$ which is closed under taking non-empty initial segments.  We will define integers $0=m_0<m_1<\ldots$ and a monotone, length preserving function $\Sigma:B\to B_Z^{<\omega}$ as part of the following recursion. 

Let $m_0=0$.

Let $Y_1=\chi(\varnothing)$. By Proposition \ref{dual1}, there exists $m_1\in\nn$ such that for any interval $I\subset (m_1, \infty)$, if $x^*\in S_{X^*}$ is such that $\|x^*-P^{\textsf{E}^*}_Ix^*\|\leqslant  \delta_1$, then there exists $z\in B_Z$ such that $Qz\in B_{Y_1}$ and $\text{Re\,}x^*(Qz)\geqslant \frac13$.  Define  $B_1=\varnothing$ and $A_1=\{\varnothing\}$. 

Next, suppose that integers $m_0<m_1<\ldots < m_n$, $Y_1, \ldots, Y_n\in \co(X)$, finite sets  $B_1, \ldots, B_n\subset S_{X^*}^{<\omega}$, and finite sets $A_1, \ldots, A_n\subset B_X$ have been defined. Suppose also that $\Sigma$ has been defined on $B_n$ and 
\[B_n = \Bigl\{(I_i, x^*_i)_{i=1}^k:k\in\nn, I_1<\ldots <I_k, I_i\subset [1, m_n] \text{\ intervals},(x^*_i)_{i=1}^k \in \prod_{i=1}^k N_{I_i,i}\Bigr\}\] 
and that for each $1\leqslant i\leqslant n$, each $x^*\in S_{X^*}$, and each interval $I\subset (m_i, \infty)$ such that $\|x^*-P^{\textsf{E}^*}_Ix^*\|\le \delta_1$, there exists $z\in B_Z$ such that $Qz\in B_{Y_i}$ and $\text{Re\,}x^*(Qz)\geqslant \frac13$.   Let \[A_{n+1}=\Bigl(\{0\}\cup \bigcup_{\sigma\in B_n}Q(\text{im}(\Sigma(\sigma)))\Bigr)^{\leqslant n},\]
\[Y_{n+1}'=\bigcap_{\sigma\in A_{n+1}}\chi(\sigma)\ \ \text{and}\ \ Y_{n+1}''=\bigcap_{x^*\in \cup_{\sigma\in B_n}\text{im}_2(\sigma)}\text{ker}(x^*),\] 
and let 
\[Y_{n+1}=Y_{n+1}'\cap Y_{n+1}'' \in \co(X).\]  
By Proposition \ref{dual1}, there exists $m_{n+1}'\in (m_n, \infty)$ such that for any interval $I\subset (m_{n+1}',\infty)$ and $x^*\in S_{X^*}$ such that $\|x^*-P^{\textsf{E}^*}_Ix^*\|\leqslant \delta_1$, there exists $z\in B_Z$ such that $Qz\in B_{Y_{n+1}}$ and $\text{Re\ }x^*(Qz)\geqslant \frac13$. Then, there exists $m_{n+1}''>m_n$ so large that for any interval $I\subset (m_{n+1}'',\infty)$ and for any $z\in \cup_{\sigma\in B_n}\text{im}(\Sigma(\sigma))$, $\|P^\textsf{E}_Iz\|<\delta_{n+1}$. Let $m_{n+1}=\max \{m_{n+1}', m_{n+1}''\}$ and let 
\[B_{n+1}=\Bigl\{(I_i, x^*_i)_{i=1}^k:k\in\nn, I_1<\ldots <I_k, I_i\subset [1, m_{n+1}] \text{ intervals},(x^*_i)_{i=1}^k \in \prod_{i=1}^k N_{I_i,i}\Bigr\}.\]  
We complete the recursion by defining $\Sigma(\sigma)$ for each $\sigma=(I_i, x^*_i)_{i=1}^k\in B_{n+1}\setminus B_n$, by induction on $k=|\sigma|$.   

Case $1$, $k=1$: 

Case $1a$, $I_1\cap [1,m_1]\neq \varnothing$: Let $\Sigma(\sigma)=(0)$. 

Case $1b$, $I_1\subset (m_1,\infty)$. Let $r$ be the minimum $1\leqslant i\leqslant n$ such that $I_1\subset (m_i, \infty)$.  Since $x^*_1\in N_{I_1,1}$, there exists $z\in B_Z$ such that $Qz\in B_{Y_r}$ and $\text{Re\ }x^*(Qz)\geqslant \frac13$. Let $\Sigma(\sigma)=(z)$. 

Case $2$, $k>1$: Assume that $\Sigma(\sigma')$ has been defined for each $\sigma'\prec \sigma$.  Let $\sigma'=(I_i, x^*_i)_{i=1}^{k-1}$. 

Case $2a$, $\sigma'\in B_{n+1}\setminus B_n$: Let $\Sigma(\sigma)=\Sigma(\sigma')\smallfrown (0)$.   

Case $2b$, $\sigma'\in B_n$, $I_k \cap [1,m_1]\neq \varnothing$: Let $\Sigma(\sigma)=\Sigma(\sigma')\smallfrown(0)$. 

Case $2c$, $\sigma'\in B_n$, $I_k\subset (m_1,\infty)$: Let $r$ be the minimum $1\leqslant i\leqslant n$ such that $I_k\subset (m_i, \infty)$.  Since $x^*_k\in S_{X^*}\cap N_{I_k,k}$ and $\delta_k\leqslant \delta_1$, there exists $z\in B_Z$ such that $Qz\in B_{Y_r}$ and $\text{Re\ }x^*(Qz)\geqslant \frac13$. Let $\Sigma(\sigma)=\Sigma(\sigma')\smallfrown(z)$.  This completes the recursive construction.

We know consider the blocking $\textsf{F}=(F_n)_{n=1}^\infty$ of $\textsf{E}$ defined by  $F_n=\oplus_{i=m_{n-1}+1}^{m_n}E_i$. Fix $1\leqslant r_1<r_2<\ldots$ in $\Ndb$, $J_1<J_2<\ldots$ intervals with $J_i\subset (r_i, r_{i+1})$, and $(x^*_i)_{i=1}^\infty\subset S_{X^*}$ such that $\|x^*_i-P^{\textsf{F}^*}_{J_i}x^*_i\|\le \delta_i$ for all $i\in\nn$.  Denote $I_i=(m_{\min J_i-1}, m_{\max J_i}]$, so that $P^{\textsf{F}^*}_{J_i}=P^{\textsf{E}^*}_{I_i}$ for all $i\in\nn$. Then $(x^*_i)_{i=1}^\infty\in \prod_{i=1}^\infty G_{I_i,i}$. Therefore there exists $(y^*_i)_{i=1}^\infty\in \prod_{i=1}^\infty N_{I_i,i}$ such that $\|x^*_i-y^*_i\|\le 4\delta_i$ for all $i\in\nn$. We have that $(I_i, y^*_i)_{i=1}^\infty\in [B]$. Let $(z_i)_{i=1}^\infty=\overline{\Sigma}((I_i, y^*_i)_{i=1}^\infty)$.    We will show that $(Qz_i, r_i)_{i=1}^\infty$ is $\chi$-admissible, $\text{Re\,}y^*_i(Qz_i)\geqslant \frac13$ for all $i\in\nn$, and $|y^*_i(Qz_i)|\leqslant 2\delta_{\max\{i,j\}}$ for all distinct $i,j$. This will finish the proof. 

Note that for any $n\in\nn$, since $I_1<\ldots<I_n$ and $I_n\subset [1, m_{\max J_n}]$ and $\max J_n<r_{n+1}$, it follows that $(I_i, y^*_i)_{i=1}^n\in B_{r_{n+1}-1}$. 

We  show that $(Qz_i, r_i)_{i=1}^\infty$ is $\chi$-admissible. Since $I_1=(m_{\min J_1-1}, m_{\max J_1}]$ and $\min J_1-1\geqslant r_1$,\[Qz_1\in B_{Y_{r_1}}\subset B_{\chi(0,\ldots, 0)},\] where the sequence in the last subscript contains $r_1-1$ zeros. Next, for $n\in\nn$, by the preceding paragraph and the definition of $A_{r_{n+1}}$, together with the fact that $\min I_{n+1}=m_{\min J_{n+1}-1}+1$ and $\min J_{n+1}-1\geqslant r_{n+1}$, \[Qz_{n+1}\in B_{Y_{r_{n+1}}}\subset B_{\chi((Qz_i, r_i)_{i=1}^n\smallfrown(0,\ldots,0)},\] where the sequence at the end of the preceding concatenation consists of $r_{n+1}-r_n-1$ zeros. This yields that $(Qz_i, r_i)_{i=1}^\infty$ is $\chi$-admissible. 

We next show that whenever $1\leqslant i<j$, $y^*_i(Qz_j)=0$ and $|y^*_j(Qz_i)|<2\delta_j$. The first equality follows from the fact that for such $i,j$, 
\[Qz_j\in Y''_{r_j}\subset \bigcap_{x^*\in \cup_{\sigma\in B_{r_j-1}\text{im}_2(\sigma)}} \ker(x^*).\] 
As noted two paragraphs above, $y^*_i\in \cup_{\sigma\in B_{r_j-1}}\text{im}_2(\sigma)$.  For the inequality, it follows from our choice of $m_{r_j}''$ and the fact that $I_j\subset (m_{r_j}'',\infty)$ and $z_i\in \cup_{\sigma\in B_{r_j}}\text{im}(\Sigma(\sigma))$ that \begin{align*} |y^*_j(Qz_i)| & \leqslant \|y^*_j-P^{\textsf{E}^*}_{I_j}y^*_j\|+|y^*_j(P^{\textsf{E}}_{I_j}z_i)|\leqslant \delta_j+\min_{1\leqslant k\leqslant r_j}\delta_k \leqslant 2\delta_j.\end{align*}

 We last show that for all $i\in\nn$, $\text{Re\ }y^*_i(Qz_i)\geqslant \frac13$. This follows from the definition of $\Sigma$ together with the fact that we are either in Case $1b$ or $2c$ for each $i$. 

\end{proof}

We can now deduce the following general result.
\begin{corollary}\label{busey} 
Let $X,Z$ be Banach spaces, $Q:Z\to X$ a quotient map, and let $\textsf{\emph{E}}$ be a shrinking FDD for $Z$.  Let $T$ be a Banach space with normalized, $1$-unconditional basis. Suppose that $c>0$ is such that Player I has a winning strategy in the $(T,c)$ game on $X$. Then there exist a blocking $\textsf{\emph{F}}$ of $\textsf{\emph{E}}$, $\Delta>0$, and a strictly decreasing sequence $(\delta_i)_{i=1}^\infty\subset (0,1)$ such that $\sum_{i=1}^\infty \delta_i<\Delta$ and such that whenever $1\leqslant r_0<r_1<\ldots$ are integers and $(x^*_i)_{i=1}^\infty\subset S_{X^*}$ satisfies $\|x^*_i-P^{\textsf{\emph{F}}^*}_{(r_{i-1},r_i)}x^*_i\|<\delta_i$ for all $i\in\nn$, then $(e^*_{r_{i}})_{i=1}^\infty\lesssim_{4c} (x^*_i)_{i=1}^\infty$. 
\end{corollary}

\begin{proof} Let $\chi:B_X^{<\omega}\to X$ be a winning strategy for Player I in the $(T,c)$ game on $X$. Let $(\delta_i)_{i=1}^\infty$ as in the previous statement and such that $\sum_{i=1}^\infty \delta_i<\Delta$, for some $\Delta>0$. Let $\textsf{{F}}$ be the blocking of $\textsf{{E}}$ given by Proposition \ref{hawk}. Let $1\leqslant r_0<r_1<\ldots$ and $(x^*_i)_{i=1}^\infty\subset S_{X^*}$ so that $\|x^*_i-P^{\textsf{\emph{F}}^*}_{(r_{i-1},r_i)}x^*_i\|<\delta_i$ for all $i\in\nn$. Let now $(y_i^*)_{i=1}^\infty \subset S_{X^*}$ and $(z_i)_{i=1}^\infty \subset B_Z$ be also given by the previous proposition. Since $\chi$ is a winning strategy for Player I, we deduce from property $(i)$ in Proposition \ref{hawk} that $(Qz_i)_{i=1}^\infty\lesssim_c (e_{r_i})_{i=1}^\infty$. Now it follows from properties $(ii)$ to $(iv)$ in Proposition \ref{hawk}, $1$-unconditionality of $(e_i)$ and elementary duality, that, for $\Delta>0$ initially chosen small enough, $(e^*_{r_{i}})_{i=1}^\infty\lesssim_{4c} (x^*_i)_{i=1}^\infty$.
\end{proof}

\section{Reducing to FDD's and pressdown norms}\label{FDD's}

The goal of this section is to prove the following crucial intermediate result.
\begin{theorem}\label{tedious} Let $p\in (1,\infty]$ and $q$ be its conjugate exponent. Let $X$ be a separable Banach space. 
\begin{enumerate}[(i)]
\item If $X$ has $\textsf{\emph{A}}_p$, then there exist $\theta\in (0,1)$ and  Banach spaces $Z,Y$ with FDDs $\textsf{\emph{F}}$, $\textsf{\emph{H}}$, respectively, such that $X$ is isomorphic to a subspace of $Z^{T^*_{q,\theta}}_\wedge(\textsf{\emph{F}})$, and to a quotient of $Y^{T^*_{q,\theta}}_\wedge(\textsf{\emph{H}})$. 
\item If $X$ has $\textsf{\emph{N}}_p$, then there exist $\theta\in (0,1)$ and  Banach spaces $Z,Y$ with FDDs $\textsf{\emph{F}}$, $\textsf{\emph{H}}$, respectively, such that $X$ is isomorphic to a subspace of $Z^{U^*_{q,\theta}}_\wedge(\textsf{\emph{F}})$, and to a quotient of $Y^{U^*_{q,\theta}}_\wedge(\textsf{\emph{H}})$. 
\end{enumerate}
\end{theorem}

Let us first state and prove the following interesting intermediate corollary.
\begin{corollary} 
Fix $1<p\leqslant \infty$ and let $X$ be a separable Banach space. Then $X$ has $\textsf{\emph{A}}_p$ (resp. $\textsf{\emph{N}}_p$)  if and only if there exists a Banach space $U$ with shrinking FDD  such that $U$ has $\textsf{\emph{A}}_p$ (resp. $\textsf{\emph{N}}_p$) and $X$ is isomorphic to a subspace and to a quotient of $Z$. 
\end{corollary}

\begin{proof} By Theorem \ref{tedious}, if $X$ has $\textsf{A}_p$, we can take $U=Z_\wedge^{T^*_{q,\theta_1}}(\textsf{F})\oplus Y_\wedge^{T^*_{q,\theta_2}}(\textsf{H})$ for appropriate Banach spaces $Z,Y$ with FDDs $\textsf{F}, \textsf{H}$, respectively, and appropriate $\theta_1, \theta_2\in (0,1)$.   For $\textsf{N}_p$, we replace the spaces $T^*_{q,\theta}$ with $U^*_{q,\theta}$. 
\end{proof}

The remainder of this section is devoted to the proof of Theorem \ref{tedious}. Before to proceed with the proof itself, we recall three already known technical statements and the corresponding references. The first lemma, based on a gliding hump argument, is classical and referred to as the Johnson-Zippin blocking lemma. Its origin can be traced back in \cite{JohnsonZippin}.

\begin{lemma}[Johnson-Zippin blocking lemma]\label{johnsonzippin}
Let $Y,Z$ be Banach spaces with boundedly-complete FDDs $\textsf{\emph{G}}, \textsf{\emph{E}}$, respectively.  Let $T:Y\to Z$ be a weak$^*$-weak$^*$-continuous operator. Then for any $(\ee_n)_{n=1}^\infty\subset (0,1)$, there exist blockings $\textsf{\emph{H}}$, $\textsf{\emph{F}}$ of $\textsf{\emph{G}}$, $\textsf{\emph{E}}$, respectively, such that for any $i<j$ and any $y\in \oplus_{n\in (i,j)}H_n$,  \[\|P^{\textsf{\emph{F}}}_{[1,i)}Ty\|\leqslant \ee_i \|y\|\ \ \ \text{and}\ \ \  \|P^{\textsf{\emph{F}}}_{[j,\infty)}Ty\|\leqslant \ee_j\|y\|.\]
\end{lemma}

The next proposition, can be found in \cite{OSZ2007} (Lemma 20) in the reflexive case. It is stated in full generality in \cite{FOSZ}. We refer to Proposition 3.12 in \cite{CauseyThesis} for a complete proof.

\begin{proposition}\label{scre} Suppose $Y,Z$ are Banach spaces with boundedly-complete FDDs $\textsf{\emph{G}}$, $\textsf{\emph{E}}$, respectively. Suppose the projection constant of $\textsf{\emph{G}}$ in $Y$ is $1$ and the projection constant of $\textsf{\emph{E}}$ is at most $K$.    Suppose $J:Y\to X$ is a weak$^*$-weak$^*$-continuous quotient map of $Y$ onto a weak$^*$-closed subspace $X$ of $Z$. Suppose also that $(\ee_i)_{i=1}^\infty\subset (0,1)$ is a strictly decreasing, null sequence such that for any $i<j$ and $y\in \oplus_{n\in (i,j)}G_n$, 
\[\|P^{\textsf{\emph{E}}}_{[1,i)}Jy\|<\frac{\ee_i \|y\|}{K}\ \ \  \text{and}\ \ \  \|P^{\textsf{\emph{E}}}_{[j,\infty)}y\|< \frac{\ee_i \|y\|}{K}.\]  
Then there exist $0=s_0<s_1<\ldots$ such that if for each $n\in\nn$, we define 
\[ C_n=\oplus_{i=s_{n-1}+1}^{s_n}G_i,\ \ \ D_n=\oplus_{i=s_{n-1}+1}^{s_n}E_i,\] 
\[L_n = \Bigl\{i\in \nn: s_{n-1}<i\leqslant \frac{s_{n-1}+s_n}{2}\Bigr\},\ \ \ R_n=\Bigl\{i\in\nn: \frac{s_{n-1}+s_n}{2} < i \leqslant s_n\Bigr\},\] \[C_{n,L}=\oplus_{i\in L_n}G_i,\ \ \ C_{n,R}=\oplus_{i\in R_n}G_i,\] then the following holds. 

For any $x\in S_X$, $0\leqslant m<n$ and $\ee>0$ such that $\|x-P^{\textsf{D}}_{(m,n)}x\|<\ee$, there exists $y\in B_Y$ with $y\in \text{span}\{C_{m,R}\cup (C_i)_{m<i<n}\cup C_{n,L}\}$, where $C_{0,R}=\{0\}$, and $\|Jy-x\|<2K\ee+6K\ee_m$. If $m=0$, we can replace this last inequality with $\|Qy-x\|<K\ee+3K\ee_1$. 
\end{proposition}

We shall also need the following (see Proposition 3.1 in \cite{FOSZ}).
\begin{proposition}\label{skip} 
Let $X$ be a Banach space, $Z$ a Banach space with shrinking FDD $\textsf{\emph{E}}$ having projection constant $K$. Assume that $Q:Z\to X$ is a quotient map and identify $X^*$ with the weak$^*$ closed subspace $Q^*(X^*)$ of $Z^*$. Let $(\delta_i)_{i=1}^\infty\subset (0,1)$ be a strictly decreasing, null sequence. Then there exist $0=s_0<s_1<\ldots$ such that for any $1\leqslant k_0<k_1<\ldots$ and $x^*\in X^*$, there exist $(x^*_i)_{i=1}^\infty \subset X^*$ and  $(t_i)_{i=1}^\infty\in \prod_{i=1}^\infty (s_{k_{i-1}-1}, s_{k_{i-1}})$ such that, with $t_0=0$, 
\begin{enumerate}[(i)]
\item $x^*=\sum_{i=1}^\infty x^*_i$, 

and for all $i\in\nn$, 

\item either $\|x^*_i\|\le \delta_i$ or $\|x^*_i-P^{\textsf{E}^*}_{(t_{i-1}, t_i)}x^*_i\|\leqslant \delta_i \|x^*_i\|$, 
\item $\|x^*_i-P^{\textsf{\emph{E}}^*}_{(t_{i-1}, t_i)}x^*\|\leqslant \delta_i$, 
\item $\|x^*_i\|\leqslant K+1$, 
\item $\|P^{\textsf{\emph{E}}^*}_{t_i}x^*\|\leqslant \delta_i$. 
\end{enumerate}
\end{proposition}

The proof of Theorem \ref{tedious} is similar to the proof of Theorem 1.1 in \cite{FOSZ}. However, given that the bases of the spaces $T^*_{q,\theta}$, $U^*_{q,\theta}$ are left dominant and not right dominant, we include the details to ease the reading and to make clear the modifications required to accommodate the replacement of right dominance by left dominance. 

\begin{proof}[Proof of Theorem \ref{tedious}] Fix $1<p\leqslant \infty$ and let $q$ be its conjugate exponent. Let $X$ be a separable Banach space with $\textsf{A}_p$ (resp. $\textsf{N}_p$). By Theorem \ref{upper1}, there exists $\theta_0 \in (0,1)$ such that Player I has a winning strategy in the spatial $(T,1)$ game, where $T=T_{q,\theta}^*$ (resp. $T=U_{q,\theta}^*$) and $\theta \in (0,\theta_0)$.

\smallskip $(1)$ We first prove that there exists a  Banach space $Z$ with shrinking FDD $\textsf{{F}}$ such that, for any for $\vartheta \in (0,\frac{\theta_0}{8}]$, $X$ is isomorphic to a quotient of $Z^S_\wedge(\textsf{{F}})$, where $S=T_{q,\vartheta}^*$ (resp. $T=U_{q,\vartheta}^*$). We now fix $\theta \in (0,\theta_0)$ and denote $\vartheta=\frac{\theta}{8}$. 

Since $X^*$ is separable, by a theorem of Davis, Figiel, Johnson, and Pe\l czy\'{n}ski (\cite{DFJP}, Corollary 8), there exists a Banach space $Z$ with shrinking FDD $\textsf{E}$ (a shrinking basis in fact) and a  bounded linear surjection $Q:Z\to X$. By first renorming $Z$ and then $X$, we can assume that $\textsf{E}$ is bimonotone in $Z$ and that $Q$ is a quotient map.  Note that $Q^*:X^*\to Z^*$ is an isometric embedding. Throughout the proof, we identify $X^*$ with its image in $Z^*$. 

By replacing $\textsf{E}$ with a blocking and then relabeling, by Corollary \ref{busey}, we can assume there exist constants $C>1$, $\Delta>0$, and $(\delta_i)_{i=1}^\infty\subset (0,1)$ with $\sum_{i=1}^\infty \delta_i < \Delta$ such that for any $1\leqslant r_0<r_1<\ldots$ and any $(x^*_i)_{i=1}^\infty\subset S_{X^*}$ such that $\|x^*_i-P^{\textsf{E}^*}_{(r_{i-1}, r_i)}x^*_i\|\leqslant \delta_i$ for all $i\in\nn$, $(e^*_{r_{i}})_{i=1}^\infty \lesssim_C (x^*_i)_{i=1}^\infty$, where $(e^*_i)_{i=1}^\infty$ is the canonical basis of $T^*$ ($=T_{q,\theta}$ or $U_{q,\theta}$).  By replacing $(\delta_i)_{i=1}^\infty$ with a smaller sequence if necessary, we can assume $(\delta_i)_{i=1}^\infty$ is strictly decreasing.

Next, suppose that $\textsf{D}$ is a blocking of $\textsf{E}$, say $D_n=\oplus_{i=j_{n-1}+1}^{j_n}E_i$.   Suppose also that $1\leqslant r_0<r_1<\ldots$ and $(I_i)_{i=1}^\infty$ are intervals and $(x^*_i)_{i=1}^\infty\subset S_{X^*}$ are such that $r_{i-1}+1=\min I_i<r_i$ and $\|x^*_i-P^{\textsf{D}^*}_{I_i}x^*_i\|\leqslant \delta_i$ for all $i\in\nn$.   By bimonotonicity, this implies that \[\|x^*_i-P^{\textsf{E}^*}_{(j_{r_{i-1}}, j_{r_i})}x^*_i\|\leqslant \delta_i,\] from which it follows that $(e^*_{j_{r_{i}}})_{i=1}^\infty\lesssim_C (x^*_i)_{i=1}^\infty$. By $1$-right dominance of the basis of $T^*$, it follows that $(e^*_{r_{i-1}})_{i=1}^\infty\lesssim_C (x^*_i)_{i=1}^\infty$. In other words, the property of $\textsf{E}$ which we just deduced from the asymptotic $T$ property is stable under passing to blockings. 

By passing to a blocking and relabeling, we can assume  that for any subsequent blocking $\textsf{D}$ of $\textsf{E}$, there exists $(f^*_i)_{i=1}^\infty\subset S_{X^*}$ such that  for all $i\in\nn$, $\|f^*_i-P^{\textsf{D}^*}_i f_i^*\|\leqslant \frac{\delta_i}{2}$.  Let $0=s_0<s_1<\ldots$ be the sequence given by Proposition \ref{skip} applied to $X$, $Z$, $\textsf{E}$, and $(\delta_i)_{i=1}^\infty$.  Let $F_n=\oplus_{i=s_{n-1}+1}^{s_n}E_i$ for all $n\in\nn$.   We claim that $Q^*$ is  an isomorphic embedding of $X^*$ into $(Z^*)^{(S^*)}_\vee(\textsf{F}^*)$, and that $Q^*$ is still weak$^*$ to weak$^*$ continuous. Here, $(Z^*)^{(S^*)}_\vee(\textsf{F}^*)$ has the weak$^*$-topology it inherits as the dual space of $Z^S_\wedge(\textsf{F})$.   From this it will follow that $Q^*$ is the adjoint of a bounded linear surjection from $Z^S_\wedge(\textsf{F})$ onto $X$, which will finish $(1)$. For the remainder of the proof, choose $(f^*_i)_{i=1}^\infty\subset S_{X^*}$ such that $\|f^*_i-P^{\textsf{F}^*}_if^*_i\|\leqslant \frac{\delta_i}{2}$ for all $i\in\nn$.

Fix $1\leqslant n_0<n_1<\ldots$ and $x^*\in S_{X^*}$.      Let $\ell_i=s_{n_i-1}$ and note that $P^{\textsf{E}^*}_{(\ell_{i-1}, \ell_i]}=P^{\textsf{F}^*}_{[n_{i-1}, n_i)}$ for all $i\in\nn$.   By our choice of $(s_i)_{i=0}^\infty$, we can find $(x^*_i)_{i=1}^\infty\subset X^*$ and $(t_i)_{i=0}^\infty\subset\nn$ with $0=t_0<t_1<\ldots$ satisfying the conclusions of Proposition \ref{skip}. 

For $i\in\nn$, if $\|x^*_{i+1}\|\geqslant \delta_{i+1}$, let $a_i=\|x_{i+1}^*\|$ and let $y^*_i=a_i^{-1}x^*_{i+1}$.  If $\|x^*_{i+1}\|<\delta_{i+1}$, let $a_i=0$ and let $y_i^*=f_{\ell_i}^*$.  Since $t_i<l_i=s_{n_i-1}<t_{i+1}$, it follows that for all $i \in \Ndb$,
\[\|y^*_i-P^{\textsf{F}^*}_{(t_i,t_{i+1})}y^*_i\|\leqslant \delta_i.\] 
Therefore $(e^*_{t_i})_{i=1}^\infty\lesssim_C (y^*_i)_{i=1}^\infty$.  Thus 
\begin{align*} 1 & = \|x^*\| = \Bigl\|\sum_{i=1}^\infty x^*_i\Bigr\| \geqslant \Bigl\|\sum_{i=1}^\infty a_iy^*_i\Bigr\|-\|x_1^*\|-\Delta \\ & \geqslant\frac{1}{C}\Bigl\|\sum_{i=1}^\infty a_ie^*_{t_i}\Bigr\|_{T^*} -2-\Delta \geqslant \frac{1}{C}\Bigl\|\sum_{i=1}^\infty \|x^*_{i+1}\|e^*_{t_i}\Bigr\|_{T^*}-2-2\Delta. \end{align*}

From this it follows that \[\Bigl\|\sum_{i=1}^\infty \|x^*_{i+1}\|e^*_{t_i}\Bigr\|_{T^*}\leqslant C(3+2\Delta).\] Moreover, 
\[\|P^{\textsf{F}^*}_{(l_{i-1}, l_i]}x\| \leqslant \|P^{\textsf{F}^*}_{(t_{i-1}, t_{i+1})}x^*\|
\leqslant \|P^{\textsf{F}^*}_{(t_{i-1}, t_{i})}x^*\|+\|P^{\textsf{F}^*}_{t_{i}}x^*\|+ \|P^{\textsf{F}^*}_{(t_{i}, t_{i+1})}x^*\|
\leqslant \|x_i^*\|+\|x^*_{i+1}\|+3\delta_i.\] 

%Therefore, 

%\begin{align*} \Bigl\|\sum_{i=1}^\infty \|P^\textsf{F}_{(\ell_{i-1}, \ell_i]}x^*\|e^*_{n_{i-1}}\Bigr\|& \leqslant \Bigl\|\sum_{i=1}^\infty \|x^*_i\|e^*_{n_{i-1}}\Bigr\| + \Bigl\|\sum_{i=1}^\infty \|x^*_{i+1}\|e^*_{n_{i-1}}\Bigr\| + 3\Delta \\ &
%\leqslant \Bigl\|\sum_{i=1}^\infty \|x^*_{i+1}\|e^*_{n_i}\Bigr\| + \Bigl\|\sum_{i=1}^\infty \|x^*_{i+1}\|e^*_{n_{i-1}}\Bigr\|+\|x^*_1\|+3\Delta \\ & \leqslant 2\Bigl\|\sum_{i=1}^\infty \|x^*_{i+1}\|e^*_{t_i}\Bigr\| + 2 + 3\Delta \\ & \leqslant 2C(3+2\Delta)+2+3\Delta.\end{align*} 

Therefore, for $S=T^*_{q,\vartheta}$ if $T=T^*_{q,\theta}$ and $S=U^*_{q,\vartheta}$ if $T=U^*_{q,\theta}$, we have

\begin{align*} \Bigl\|\sum_{i=1}^\infty \|P^{\textsf{F}^*}_{(\ell_{i-1}, \ell_i]}x^*\|e^*_{n_{i-1}}\Bigr\|_{S^*} & \leqslant \Bigl\|\sum_{i=1}^\infty \|x^*_i\|e^*_{n_{i-1}}\Bigr\|_{S^*} + \Bigl\|\sum_{i=1}^\infty \|x^*_{i+1}\|e^*_{n_{i-1}}\Bigr\|_{S^*} + 3\Delta \\ &
\leqslant \Bigl\|\sum_{i=1}^\infty \|x^*_{i+1}\|e^*_{s_{n_i}}\Bigr\|_{S^*} + \Bigl\|\sum_{i=1}^\infty \|x^*_{i+1}\|e^*_{s_{n_{i-1}}}\Bigr\|_{S^*} +\|x^*_1\|+3\Delta \\ &
 \leqslant 3^3\Bigl\|\sum_{i=1}^\infty \|x^*_{i+1}\|e^*_{t_i}\Bigr\|_{T^*}+3\Bigl\|\sum_{i=1}^\infty \|x^*_{i+1}\|e^*_{t_i}\Bigr\|_{T^*} + 2 + 3\Delta \\ & \leqslant [3^3+3]C(3+2\Delta)+2+3\Delta=M.\end{align*}
The first inequality follows from Proposition \ref{skip}, the second from the $1$-right dominance of the canonical basis of $S^*$ and the third from Lemma \ref{shuffle} and the remark after it  with $m=0$ and $m=1$, using the fact that $t_i\in (s_{n_{i-1}-1}, s_{n_{i-1}})$ for all $i\in\nn$, so $t_1<s_{n_0}<t_2<s_{n_1}<\ldots$. This shows that for all $x^* \in X^*$, $\|x^*\|_{(Z^*)^{(S^*)}_\vee(\textsf{F}^*)}\leqslant M\|x^*\|$. Of course $\|x^*\|_{(Z^*)^{(S^*)}_\vee(\textsf{F}^*)}\geqslant \|x^*\|$ and this finishes the proof of the fact that $Q^*$ is an embedding from $X^*$ into $(Z^*)^{(S^*)}_\vee(\textsf{F}^*)$ 

Since the proof of the weak$^*$ to weak$^*$ continuity of $Q^*$ in \cite{FOSZ} did not use right dominance or block stability, it goes through unchanged.

\smallskip $(2)$ We now prove that there exists a  Banach space $Z$ with shrinking FDD $\textsf{{H}}$ such that, for any for $\vartheta \in (0,\frac{\theta_0}{4}]$, $X$ is isomorphic to a a subspace of $Z^S_\wedge(\textsf{\emph{H}})$, where $S=T_{q,\vartheta}^*$ (resp. $S=U_{q,\vartheta}^*$). We now fix $\theta \in (0,\theta_0)$ and denote $\vartheta=\frac{\theta}{4}$. 

Since $X^*$ is separable we use again the result of Davis, Figiel, Johnson, and Pe\l cyz\'{n}ski \cite{DFJP} insuring the existence of a Banach space $Z$ with shrinking FDD $\textsf{E}$ and a quotient map $Q:Z\to X$.  By Lemma $3.1$ of \cite{OS2002}, there exist a Banach space $Y$ with shrinking FDD $\textsf{G}$ and an isomorphic embedding $\iota:X\to Y$ such that $c_{00}(\textsf{G})\cap X$ is dense in $X$ (identified with its image $\iota(X))$.  By first renorming $Y$, then $X$, then $Z$, we can assume that $\textsf{G}$ is bimonotone in $Y$, that $\iota$ is an isometric embedding, and that $Q$ is still a quotient map.  We consider $X^*$ as a subspace of $Z^*$ and we consider $\iota^*$ as mapping $Y^*$ to either $X^*$ or $Z^*$, as is convenient. Let $K$ be the projection constant of $\textsf{E}$ in $Z$. 

Since Player I has a winning strategy in the $(T,1)$ game on $X$, by Corollary \ref{busey}, there exist a blocking, which we can assume after relabling is $\textsf{E}$,  $C\geqslant 1$, $\Delta>0$, and a strictly decreasing sequence $(\delta_i)_{i=1}^\infty\subset (0,1)$ such that $\sum_{i=1}^\infty \delta_i<\Delta$, and if $(x^*_i)_{i=1}^\infty\subset S_{X^*}$ is such that $\|x^*_i-P^{\textsf{E}^*}_{(r_{i-1}, r_i)}x^*_i\|\leqslant 2K\delta_i$ for all $i\in\nn$, then $(e^*_{r_{i}})_{i=1}^\infty\lesssim_C (x^*_i)_{i=1}^\infty$.   By replacing $(\delta_i)_{i=1}^\infty$ smaller if necessary, we can assume that for any  $(r_i)_{i=0}^\infty$ and $(x^*_i)_{i=1}^\infty$ as above,  and if $(z^*_i)_{i=1}^\infty\subset Z^*$ satisfies $\|x^*_i-z^*_i\|\leqslant \delta_i$ for all $i\in\nn$, then $(z^*_i)_{i=1}^\infty$ is basic with projection constant not more than $2K$. We can also assume that $\sum_{i=1}^\infty \delta_i<\frac17$.     

Let us observe that if $\textsf{D}$ is any further blocking of $\textsf{E}$, say $D_n=\oplus_{i=j_{n-1}+1}^{j_n}E_i$, and if $(x^*_i)_{i=1}^\infty\subset S_{X^*}$ and $1\leqslant r_0<r_1<\ldots$ are such that $\|x^*_i-P^{\textsf{D}^*}_{(r_{i-1}, r_i)}x^*_i\|\leqslant \delta_i$, then with $m_i=j_{r_i}$, it follows that $\|x^*_i-P^{\textsf{E}}_{(j_{r_{i-1}}, j_{r_i})}x^*_i\|\leqslant 2K\delta_i$, and $(e^*_{r_{i-1}})_{i=1}^\infty\lesssim_1 (e^*_{j_{r_{i}}})_{i=1}^\infty\lesssim_C (x^*_i)_{i=1}^\infty$. Here we have used $1$-right dominance of the basis of $T^*$.  We will use this fact as we pass to further blockings in the proof. 

Fix a strictly decreasing sequence $(\ee_i)_{i=1}^\infty \subset (0,1)$ such that for each $n\in\nn$, \[10K(K+1)\sum_{i=n}^\infty \ee_i < \delta_n^2.\] 

After blocking $\textsf{E}$ if necessary, we may assume that for that for each further blocking $\textsf{D}$ of $\textsf{E}$, there exists $(f^*_i)_{i=1}^\infty\subset S_{X^*}$ such that for each $i\in\nn$, $\|f^*_i-P^{\textsf{D}^*}_i f^*_i\|<\frac{\ee_{i+1}}{2K}$. After blocking $\textsf{G}$, we can assume that for each $i\in\nn$, $\iota^*(\textsf{G}^*_i)\neq \{0\}$.   

Using Lemma \ref{johnsonzippin}, after blocking and relabeling $\textsf{E}$ and $\textsf{G}$, we can assume that for each $i<j$ and each $y^*\in \oplus_{n\in(i,j)} \textsf{G}^*_n$, 
\[\|P^{\textsf{E}^*}_{[1,i)}\iota^* y^*\|<\frac{\ee_i \|y^*\|}{K}\ \ \text{and}\ \  \|P^{\textsf{E}^*}_{[j,\infty)}\iota^* y^*\|< \frac{\ee_i \|y^*\|}{K},\]  and that this property is preserved if we pass to any further blocking of one FDD and the corresponding blocking of the other. 

Let $\textsf{C}, \textsf{D}$ be the blockings of $\textsf{G}, \textsf{E}$, respectively, determined by Proposition \ref{scre} applied with the sequence $(\ee_i)_{i=1}^\infty$. More precisely, Proposition \ref{scre} is appliable to $\textsf{G}^*$ and $\textsf{E}^*$ rather than $\textsf{G}$ and $\textsf{E}$, but we actually apply the proposititon to the dual FDDs and let $\textsf{C}, \textsf{D}$ be the corresponding blockings of the original FDDs. Apply Proposition \ref{skip} to the FDD $\textsf{D}$ with the sequence $(\ee_i)_{i=1}^\infty$ to obtain $0=s_0<s_1<\ldots$.  Define $H_n=\oplus_{i=s_{n-1}+1}^{s_n} C_i$ and $F_n=\oplus_{i=s_{n-1}+1}^{s_n}D_i$ for all $n\in\nn$.    For $n\in\nn$, define $\tilde{H^*}_n=H_n^*/\ker(\iota^*|_{H_n^*})$, endowed with the quotient norm \[\|\tilde{y^*}\|_\sim = \|\iota^* y^*\|.\] Note that $\tilde{H^*}_n\neq \{0\}$, since for each $n\in\nn$, $\iota^*(G_n^*)\neq \{0\}$.   Given $y^*=\sum_{n=1}^\infty y_n^*\in c_{00}(\textsf{H}^*)$, we set $\tilde{y^*}=\sum_{n=1}^\infty\tilde{y^*}_n\in c_{00}(\tilde{H}^*)$. We set \[\|\tilde{y^*}\|_\sim = \max_{i\leqslant j} \Bigl\|\iota^*\Bigl(\sum_{n=i}^j y_n^*\Bigr)\Bigr\| = \max_{i\leqslant j}\|\iota^*P^{\textsf{H}^*}_{[i,j]}y^*\|.\]  Clearly $\tilde{H^*}$ is a bimonotone FDD for the completion $W$ of $c_{00}(\tilde{H^*})$ under  $\|\ \|_\sim$. 

Note that $y\mapsto\tilde{y}$ extends to a norm $1$ operator from $Y^*$ into $W$.  By the definition of $\|\cdot\|_\sim$, $\|\iota^* y^*\|\leqslant \|\tilde{y^*}\|_\sim$ for all $y\in c_{00}(\textsf{H}^*)$.   Thus $\tilde{y^*}\mapsto \iota^*y^*$ extends to a norm $1$ operator $\tilde{\iota^*}:W\to X^*$. Moreover, $\iota^* y^*=\tilde{\iota^*}\tilde{y^*}$ for all $y^*\in W$.    

The proof of the following claim is unchanged from \cite{FOSZ}, so we omit it.

\begin{claim}\ 
\begin{enumerate}[(i)]
\item $\tilde{\iota^*}$ is a quotient map. More precisely, if $x^*\in S_{X^*}$ and $y^*\in S_{Y^*}$ are such that $\iota^* y^*=x^*$, then $\tilde{y^*}\in S_W$ and $\tilde{\iota^*}y^*=x^*$. 
\item If $(\tilde{y^*}_n)_{n=1}^\infty$ is a subnormalized block sequence in $W$ with respect to $\tilde{\textsf{\emph{H}}}$ such that $(\iota^* \tilde{y^*_n})_{n=1}^\infty$ is basic with projection constant not more than $\tilde{K}$ and $a=\inf_n \|\tilde{\iota^*}\tilde{y^*_n}\|>0$, then for all $(a_n)_{n=1}^\infty\in c_{00}$, \[\Bigl\|\sum_{n=1}^\infty a_n \tilde{\iota^*}\tilde{y^*_n}\Bigr\|\leqslant \Bigl\|\sum_{n=1}^\infty a_n \tilde{y^*_n}\Bigr\|_\sim \leqslant \frac{3\tilde{K}}{a}\sum_{n=1}^\infty \Bigl\|\sum_{n=1}^\infty a_n \tilde{\iota^*} \tilde{y^*_n}\Bigr\|.\] 
\end{enumerate}
\label{arq}
\end{claim}

Since the inclusion of $W$ in $W^{S^*}_\vee(\tilde{\textsf{H}^*})$ is of norm $1$, we can consider $\tilde{\iota^*}$ as an operator from $W^{T^*}_\vee(\tilde{\textsf{H}^*})$ into $X^*$. We will prove the following.

\begin{claim}\label{claim2} There exists a constant $A>0$  such that for any $x^*\in S_{X^*}$, there exists $\tilde{y^*}\in W_\vee^{S^*}(\tilde{\textsf{H}^*})$ with $\|\tilde{y^*}\|_\vee \leqslant A$ such that $\|\tilde{\iota^*}\tilde{y^*}-x^*\|<\frac12$.
\end{claim}
This will yield that $\tilde{\iota^*}:W^{S^*}_\vee(\tilde{\textsf{H}^*})\to X^*$ is a surjection. Then, $W^{S^*}_\vee(\tilde{\textsf{H}^*})$ is naturally the dual space of $(W^{(*)})_\wedge^S(\tilde{\textsf{H}^*}^*)$, and as such naturally has a weak$^*$-topology determined by this predual. We then conclude the proof by arguing that $\tilde{\iota^*}$ is weak$^*$ to weak$^*$ continuous, and is therefore the adjoint of an embedding of $X$ into  $(W^{(*)})_\wedge^S(\tilde{\textsf{H}^*}^*)$. As in part $(1)$, the argument that $\tilde{\iota^*}$ is weak$^*$ to weak$^*$ continuous goes through unchanged from \cite{FOSZ}, so we omit it.

We now conclude with the proof of Claim \ref{claim2}. Fix $x^*\in S_{X^*}$ and $(f^*_i)_{i=1}^\infty \subset S_{X^*}$ such that for each $i\in\nn$, $\|f^*_i - P^{\textsf{D}^*}_i f_i^*\|<\frac{\ee_{i+1}}{2K}$.    Fix $(x^*_i)_{i=1}^\infty\subset X^*$ and $0=t_0<t_1<\ldots$ according to Proposition \ref{skip} such that for each $i\in\nn$, $t_i\in (s_{i-1}, s_i)$, $x^*=\sum_{i=1}^\infty x^*_i$, and either $\|x^*_i\|\le \ee_i$ or $\|x^*_i-P^{\textsf{D}^*}_{(t_{i-1}, _i)}x^*_i\|\le \ee_i\|x^*_i\|$.     

If $\|x_{i+1}^*\|> \ee_{i+1}$, let $w^*_i=\|x^*_{i+1}\|^{-1}x^*_{i+1}$ and $a_i=\|x_{i+1}^*\|$. If $\|x^*_{i+1}\|<\ee_{i+1}$, let $w^*_i=f^*_{s_i}$ and let $a_i=0$.  Note that $\|w^*_i-P^{\textsf{D}^*}_{(t_i, t_{i+1})}w^*_i\|<\ee_{i+1}$ for all $i\in\nn$, so $(e^*_{t_i})_{i=1}^\infty \lesssim (e^*_{t_{i+1}})_{i=1}^\infty \lesssim_C (w^*_i)_{i=1}^\infty$.  By Proposition \ref{scre}, there exists a sequence $(y^*_i)_{i=1}^\infty \subset B_{Y^*}$ with \[y^*_i\in \text{span}\{C_{t_i,R}^*\cup (C_j^*)_{t_i<j<t_{i+1}}\cup C^*_{t_{i+1},L}\}\] such that 
\[\|\iota^*y^*_i-w_i^*\|\leqslant 2K\ee_{i+1}+6K\ee_i\leqslant  4K(K+1)\sum_{j=i}^\infty \ee_j<\delta_i.\] 

If $\|x^*_1\|\leqslant \ee_1$, let $y^*_0=0$. Otherwise, we use Proposition \ref{scre} again to find $y^*_0\in Y^*$ with $\|y^*_0\|<K+1$ such that \[y^*_0\in \text{span}\{(C^*_i)_{i<t_1}\cup C^*_{t_1,L}\}\] and \[\|\iota^* y^*_0- x^*_1\|<4K\ee_1\|x^*_1\|<4K(K+1)\ee_1.\]    
Set $w^*=x_1^*+\sum_{i=1}^\infty a_iw^*_i$. Note that this series converges and \[\|x^*-w^*\|\leqslant \sum_{i=2}^\infty \ee_i<\frac14.\]   
By our choice of $(\delta_i)_{i=1}^\infty$, and since $\|\tilde{\iota^*}\tilde{y^*_i}-w^*_i\|<\delta_i$ for all $i\in\nn$, $(\tilde{\iota^*}\tilde{y^*}_i)_{i=1}^\infty$ is basic with projection constant at most $2K$ and equivalent to $(w^*_i)_{i=1}^\infty$. Furthermore, \[\inf_i \|\tilde{\iota^*} \tilde{y^*}_i\|\geqslant \inf_i (\|w^*_i\|-\delta_i)> \frac67.\]   By Claim \ref{arq}, \[\Bigl\|\sum_{i=1}^\infty c_i \tilde{\iota^*}\tilde{y^*}_i\Bigr\|\leqslant \Bigl\|\sum_{i=1}^\infty \tilde{y^*}_i\Bigr\|\leqslant 7K \Bigl\|\sum_{i=1}^\infty c_i \tilde{\iota^*}\tilde{y^*}_i\Bigr\|\] for any $(c_i)_{i=1}^\infty \in c_{00}$.  Thus $(\tilde{y^*}_i)_{i=1}^\infty$ is basic, equivalent to $(w^*_i)_{i=1}^\infty$, and $\sum_{i=1}^\infty a_i \tilde{y^*}_i$ converges. 

We recall that we have fixed $\theta \in (0,\theta_0]$ and denoted $\vartheta=\frac{\theta}{4}$, $T=T^*_{q,\theta}$ (resp. $T=U^*_{q,\theta}$) and  $S=T^*_{q,\vartheta}$ (resp. $S=U^*_{q,\vartheta}$).

Let $\tilde{y^*}=y^*_0+\sum_{i=1}^\infty a_i\tilde{y^*}_i$. We have \begin{align*} \|\tilde{\iota^*}\tilde{y^*}-w^*\| & \leqslant \|\tilde{\iota^*}\tilde{y^*}_0-x^*_1\| + \sum_{i=1}^\infty|a_i| \|\tilde{\iota^*}\tilde{y^*}_i-w^*_i\|  \leqslant 10K(K+1)\sum_{i=1}^\infty \ee_i <\frac14. \end{align*} Thus $\|\tilde{\iota^*}\tilde{y^*}-x^*\|<\frac12$. 

We next prove the norm estimate.  Fix $1\leqslant n_0<n_1<\ldots$. Note that $\tilde{y^*}_i\in \tilde{H}_i\oplus\tilde{H}_{i+1}$ and $\tilde{y^*}_0\in \tilde{H}_1$. It follows that \begin{align*} \Bigl\|\sum_{i=1}^\infty \|P^{\tilde{\textsf{H}^*}}_{[n_{i-1}, n_i)}\tilde{y^*}\|_\sim e^*_{n_{i-1}}\Bigr\|_{S^*} & \leqslant \|\tilde{y^*}_0\|_\sim + \Bigl\|\sum_{i=1}^\infty a_{n_{i-1}-1}e_{n_{i-1}}\Bigr\|_{S^*} \\ & + \Bigl\|\sum_{i=1}^\infty \bigl\|\sum_{j=n_{i-1}}^{n_i-1}a_j\tilde{y^*_j}\bigr\|_\sim e_{n_{i-1}}\Bigr\|_{S^*}, \end{align*} where we put $a_0=0$ if $n_0=1$.  This is because $\tilde{y^*}_0$ may have non-zero image only under the first projection $P^{\tilde{\textsf{H}^*}}_{[n_0, n_1)}$, which accounts for the first term on the right.    If $j\in [n_{i-1}, n_i)$, $\tilde{y^*}_i$ may have non-zero image only under the projection $P^{\tilde{\textsf{H}^*}}_{[n_{i-1}, n_i)}$.  For $i\geqslant 1$, $\tilde{w}_{n_{i-1}-1}$ may have non-zero image under either $P^{\tilde{\textsf{H}^*}}_{[n_{i-1}, n_i)}$ or $P^{\tilde{\textsf{H}^*}}_{[n_i, n_{i+1})}$. The images $P^{\tilde{\textsf{H}^*}}_{[n_i, n_{i+1})}\tilde{y^*}_{n_{i-1}-1}$ account for the second term on the right. 

%We first compute 
%\begin{align*} \Bigl\|\sum_{i=1}^\infty a_{n_{i-1}-1}e_{n_{i-1}}\Bigr\| & \leqslant \Bigl\|\sum_{i=1}^\infty a_i e_{i+1}\Bigr\| \leqslant \Bigl\|\sum_{i=1}^\infty a_ie_{t_i}\Bigr\| \\ & \leqslant C \Bigl\|\sum_{i=1}^\infty a_i w_i^*\Bigr\| = C \|w^*-x^*_1\| \\ & \leqslant C\bigl[\|x^*\|+\|x^*_1\|+\|x^*-w^*\|\bigr] < C(K+3).
%\end{align*}

We first compute, applying Lemma \ref{shuffle} for the second line,
\begin{align*} 
\Bigl\|\sum_{i=1}^\infty a_{n_{i-1}-1}e_{n_{i-1}}\Bigr\|_{S^*} 
& \leqslant \Bigl\|\sum_{i=1}^\infty a_i e_{i+1}\Bigr\|_{S^*} \leqslant \Bigl\|\sum_{i=1}^\infty a_ie_{t_i}\Bigr\|_{S^*} \leqslant |a_1|+\Bigl\|\sum_{i=2}^\infty a_ie_{t_i}\Bigr\|_{S^*} \\ 
& \leqslant K+1+9\Bigl\|\sum_{i=2}^\infty a_ie_{t_{i-1}}\Bigr\|_{T^*}
\\ & \leqslant K+1+9C\Bigl\|\sum_{i=2}^\infty a_iw^*_i\Bigr\| 
\\ & \leqslant (9C+1)(K+1)+9C\Bigl\|\sum_{i=1}^\infty a_iw^*_i\Bigr\| 
\\ & = (9C+1)(K+1) + 9C \|w^*-x^*_1\| \\ & \leqslant (9C+1)(K+1)+9C\bigl[\|x^*\|+\|x^*_1\|+\|x^*-w^*\|\bigr] \\ &  < (9C+1)(K+1)+ 9C(K+3).
\end{align*}

For each $i\in\nn$, let 
\[h_i^* = \sum_{j=n_{i-1}}^{n_i-1} a_j \tilde{y^*}_j,\ \ \text{and}\ \ g_i^* = \sum_{j=n_{i-1}}^{n_i-1}a_j w_j^*.\]   
First note that $\|h_i^*\|_\sim \leqslant 7K \|\tilde{\iota^*}h_i^*\|$. Next, \begin{align*} \|g_i^*-P^{\textsf{D}^*}_{(t_{n_{i-1}}, t_{n_i})}g_i^*\| & \leqslant \sum_{j=n_{i-1}}^{r_i-1} |a_j| 2K \|w^*_j - P^{\textsf{D}^*}_{(t_i, t_{i+1})}\| \\ & < 2K(K+1) \sum_{j=r_{i-1}}^\infty \ee_j <\delta_i^2.\end{align*}

If $\|g_i^*\|> \delta_i$, let $u_i^*=\|g_i^*\|^{-1}g_i^*$ and $b_i=\|g_i^*\|$.   Otherwise let $u_i^* = y_{n_{i-1}}$ and $b_i=0$.      Then $(u_i^*)_{i=1}^\infty\subset S_{X^*}$ is such that \[\|u^*_i-P^{\textsf{D}^*}_{(t_{n_{i-1}}, t_{n_i})}u^*_i\|<\delta_i.\]  This means $(u^*_i)_{i=1}^\infty$ is a basic sequence with projection constant not more than $2K$. Then 
\begin{align*} \Bigl\|\sum_{i=1}^\infty \|h_i^*\|_\sim e_{n_{i-1}}\Bigr\|_{S^*} & \leqslant \Bigl\|\sum_{i=1}^\infty \|h_i^*\|_\sim e_{n_{i-1}}\Bigr\|_{T^*}  \leqslant 7K\Bigl\|\sum_{i=1}^\infty \|\tilde{\iota^*}h^*_i\|e_{n_{i-1}}\Bigr\|_{T^*} \\ & \leqslant 7K\Bigl\|\sum_{i=1}^\infty \|g_i^*\|e_{n_{i-1}}\Bigr\|_{T^*} + 7K\Delta  \leqslant 7K\Bigl\|\sum_{i=1}^\infty b_i e_{n_{i-1}}\Bigr\|_{T^*} + 14K\Delta \\ & \leqslant 7K\Bigl\|\sum_{i=1}^\infty b_ie_{t_{n_{i-1}}}\Bigr\|_{T^*} + 14K\Delta  \leqslant 7CK\Bigl\|\sum_{i=1}^\infty b_i u_i^*\Bigr\| + 14K\Delta 
\\ & \leqslant 7CK\Bigl\|\sum_{i=n_0}^\infty a_iw_i^*\Bigr\| + 14K\Delta(C+1)  \\ & \leqslant 14CK^2 \Bigl\|\sum_{i=1}^\infty a_iw_i^*\Bigr\| + 14K\Delta(C+1) 
\\ & \leqslant 14CK^2(K+3)+14K\Delta(C+1):=A.  \end{align*}
This finishes the proof that $\|\tilde{y}\|_\wedge \le A$.

\end{proof}

%\end{proof}

\section{Small universal families}\label{universal}

The final step, as in \cite{FOSZ}, is to use the complementably universal space for Banach spaces with an FDD built by Schechtman in \cite{Schechtman1975}, who proved the existence of a Banach space $W$ with bimonotone FDD $\textsf{J}$ (where $\textsf{J}$ is a sequence of finite dimensional normed spaces which is dense in the space of all finite dimensional normed spaces for the Banach-Mazur distance) such that if $Z$ is any Banach space with bimonotone FDD $\textsf{H}$, then there exist a sequence of integers $m_1<m_2<\ldots$ and a bounded, linear operator $A:Z\to W$ such that $A(H_n)=J_{m_n}$ for all $n\in \Ndb$  and 
	$$\forall z \in Z,\ \ \frac12\|z\|_Z \leqslant \|Az\|_W \leqslant 2\|z\|_Z$$ and such that $A(Z)=\overline{\text{span}}\{J_{m_n}:n\in\nn\}$ is $1$-complemented in $W$ via the map $P:w\mapsto \sum_{n=1}^\infty P^\textsf{J}_{m_n}w$.

For $M=(m_n)_{n=1}^\infty\in[\nn]^\omega$, $1\leqslant q<\infty$, and $0<\theta<1$, we refer to subsection \ref{modelspaces} for the definition of the spaces $T_{M,q,\theta}$ and $U_{M,q,\theta}$. Then we have the following. 

\begin{proposition}\label{Schechtman} Let $Z$ be a Banach space with FDD $\textsf{\emph{H}}$.   Let $A:Z\to W$ and $m_1<m_2<\ldots$ be the operator and the sequence given by Schechtman's theorem exactly as in the introductory paragraph. Fix $1<p\leqslant \infty$ and let $q$ be its conjugate exponent.  Then if $(T,U)$ is either of the pairs 
\begin{enumerate}[(i)]
\item $(T_{q,\theta}^*, T_{M,q,\theta}^*)$, 
\item $(U_{q,\theta}^*, U_{M,q,\theta}^*)$, 
 \end{enumerate} then $A:c_{00}(\textsf{\emph{H}})\to c_{00}(\textsf{\emph{J}})$ extends to an isomorphic embedding $\tilde{A}:Z^T_\wedge(\textsf{\emph{H}})\to W^U_\wedge(\textsf{\emph{J}})$, the range of which is still $1$-complemented in $W^U_\wedge(\textsf{\emph{J}})$ by $P$. 
\end{proposition}

\begin{proof} 
Let $(T,U)$ be one of the indicated pairs.   Let us denote $(e_n)_{n=1}^\infty$ the canonical basis of $c_{00}$.  Note that $(e_n)_{n=1}^\infty\subset T$ is isometrically equivalent to $(e_{m_n})_{n=1}^\infty \subset U$. 

Fix $z\in c_{00}(\textsf{H})$.    Fix intervals $I_1<I_2<\ldots$ such that $\nn=\cup_{n=1}^\infty I_n$.   Let $J_n=[m_{\min I_n}, m_{\min I_{n+1}})$.  Let $J_0=[1,m_1)$, which is empty if $m_1=1$. Note that $J_0Az=0$.   Then since $AI_nz=J_nAz$ for all $n\in\nn$, 
\begin{align*} \Bigl\|\sum_{n=1}^\infty \|I_nz\|_Z e_{\min I_n}\Bigr\|_T & = \Bigl\|\sum_{n=1}^\infty \|I_nz\|_Z e_{m_{\min I_n}}\Bigr\|_U = \Bigl\|\sum_{n=1}^\infty \|I_nz\|_Z e_{\min J_n}\Bigr\|_U \\ 
& \geqslant \frac12 \Bigl\|\sum_{n=0}^\infty \|J_nAz\|_Y e_{\min J_n}\Bigr\|_U \geqslant \frac12 [Az]_\wedge. 
\end{align*} 
Taking the infimum over such $(I_n)_{n=1}^\infty$ yields that $[z]_\wedge \geqslant \frac12[Az]_\wedge$. Now,  
\begin{align*} 
\|z\|_\wedge & = \inf\Bigl\{\sum_{i=1}^n [z_i]_\wedge:n\in\nn, z_i\in c_{00}(\textsf{H}), z=\sum_{i=1}^n z_i\Bigr\}\\ 
& \geqslant \frac12\inf\Bigl\{\sum_{i=1}^n [Az_i]_\wedge:n\in\nn, z_i\in c_{00}(\textsf{H}), z=\sum_{i=1}^n z_i\Bigr\} \\ 
& \geqslant \frac12\inf\Bigl\{\sum_{i=1}^n [w_i]_\wedge:n\in\nn, w_i\in c_{00}(\textsf{J}), Az=\sum_{i=1}^n w_i\Bigr\} = \frac12\|Az\|_\wedge.  
\end{align*}
So $A$ extends to a bounded operator $\tilde{A}:Z^T_\wedge(\textsf{{H}})\to W^U_\wedge(\textsf{{J}})$ of norm at most $2$. 

Let $P:c_{00}(\textsf{J})\to c_{00}(\textsf{J})$ be given by $Pw=\sum_{n=1}^\infty P^{\textsf{J}}_{m_n}w$.   Fix $w\in c_{00}(\textsf{J})$ and intervals $I_1<I_2<\ldots$ such that $\nn=\cup_{n=1}^\infty I_n$.    Then \[\Bigl\|\sum_{n=1}^\infty \|I_nPw\|_W e_{\min I_n}\Bigr\|_U  = \Bigl\|\sum_{n=1}^\infty \|PI_nw\|_W e_{\min I_n}\Bigr\|_U \leqslant \Bigl\|\sum_{n=1}^\infty \|I_nw\|_W e_{\min I_n}\Bigr\|_U. \] Taking the infimum over such $(I_n)_{n=1}^\infty$ yields that $[Pw]_\wedge \leqslant [w]_\wedge$ for all $w\in c_{00}(\textsf{J})$. From this, the equality $PA=A$ and the injectivity of $A$, it follows that for $z\in c_{00}(\textsf{H})$, 
\begin{align*}
\|Az\|_\wedge & = \inf\Bigl\{\sum_{i=1}^n [w_i]_\wedge: n\in\nn, w_i\in c_{00}(\textsf{J}), Az=\sum_{i=1}^n w_i\Bigr\} \\ 
&  \geqslant \inf\Bigl\{\sum_{i=1}^n [Pw_i]_\wedge: n\in\nn, w_i\in c_{00}(\textsf{J}), Az=\sum_{i=1}^n w_i\Bigr\} \\ 
& = \inf\Bigl\{\sum_{i=1}^n [w'_i]_\wedge: n\in\nn, w_i\in c_{00}(\textsf{J})\cap P(W), Az=\sum_{i=1}^n w'_i\Bigr\}\\
&  \geqslant \inf\Bigl\{\sum_{i=1}^n [Az_i]_\wedge:n\in\nn, z_i\in c_{00}(\textsf{H}), z=\sum_{i=1}^n z_i\Bigr\}
\end{align*}  
The other inequality being obvious, we get
\begin{align*} \|Az\|_\wedge = \inf\Bigl\{\sum_{i=1}^n [Az_i]_\wedge:n\in\nn, z_i\in c_{00}(\textsf{H}), z=\sum_{i=1}^n z_i\Bigr\}.\end{align*}

With $z\in c_{00}(\textsf{H})$ still fixed, choose intervals $J_1<J_2<\ldots$ such that $\cup_{n=1}^\infty J_n=\Ndb$. Let 
\[N=\{n_1<\cdots <n_i<\cdots\} = \{n\in\nn:J_n\cap M\neq \varnothing\}.\]  
For $n\in N$, let $J_n'=[\min J_n, \min (J_n\cap M))$ and let $J_n''=[\min (J_n\cap M), \min J_{n+1})$. Note that $J_n'=\varnothing$ if $\min J_n = \min (J_n\cap M)$.   For $n\in\nn\setminus N$, let $J_n'=\varnothing$ and let $J_n''=J_n$.   For all $n\in\nn$, let  $K_n=J_n''\cup J_{n+1}'$ and note that $\cup_{n=n_1}^\infty K_n=[m_1, \infty)$. Note also that $\min K_n\geqslant \min J_n$ and $K_nAz=J_nAz$ for all $n\in\nn$.  For each $n\in\nn$, define $I_n=\{i\in\nn:m_i\in K_n\}$ and note that $I_1<I_2<\ldots$,  $\nn=\cup_{n=1}^\infty I_n$, and $m_{\min I_n}=\min K_n$ for all $n\in\nn$.    Moreover, $AI_nz=K_nAz$ for all $n\in\nn$.    Therefore 
\begin{align*} 
[z]_\wedge & \leqslant \Bigl\|\sum_{n=1}^\infty \|I_nz\|_Z e_{\min I_n}\Bigr\|_T = \Bigl\|\sum_{n=1}^\infty \|I_nz\|_Z e_{m_{\min I_n}}\Bigr\|_U = \Bigl\|\sum_{n=1}^\infty \|I_nz\|_Z e_{\min K_n}\Bigr\|_U \\ 
& \leqslant 2 \Bigl\|\sum_{n=1}^\infty \|K_nAz\|_W e_{\min K_n}\Bigr\|_U = 2 \Bigl\|\sum_{n=1}^\infty \|J_nAz\|_W e_{\min K_n}\Bigr\|_U. 
\end{align*} 
Taking the infimum over such $(J_n)_{n=1}^\infty$ yields that $[z]_\wedge \leqslant 2 [Az]_\wedge$ for any $z\in c_{00}(\textsf{H})$.  Finally,  
\begin{align*} 
\|z\|_\wedge & = \inf\Bigl\{\sum_{i=1}^n [z_i]_\wedge : n\in\nn, z_i\in c_{00}(\textsf{H}), z=\sum_{i=1}^n z_i\Bigr\} \\ 
& \leqslant 2\inf\Bigl\{\sum_{i=1}^n [Az_i]_\wedge : n\in\nn, z_i\in c_{00}(\textsf{H}), z=\sum_{i=1}^n\Bigr\} = 2\|Az\|_\wedge. 
\end{align*}
This shows that $\tilde{A}$ is an isomorphic embedding of $Z^T_\wedge(\textsf{H})$ into $W^U_\wedge(\textsf{J})$.  

Finally, note that we can essentially repeat the argument for $\|Az\|_\wedge$, to show that 
\[\|Pw\|_\wedge = \inf\Bigl\{\sum_{i=1}^n [Pw_i]_\wedge:n\in\nn, w_i\in c_{00}(\textsf{J}), w=\sum_{i=1}^n w_i\Bigr\}.\] We then use the inequality $[Pw]_\wedge \leqslant [w]_\wedge$ to deduce that $P$ is still a norm $1$ projection onto the closure of $c_{00}(\textsf{K})$ in $Z^U_\wedge(\textsf{J})$, where $K_n=J_{m_n}$. 

\end{proof}

We can now conclude the construction of our small universal families.
\begin{theorem} Fix $1<p\leqslant \infty$ and let $q$ be its conjugate exponent. Let $X$ be a separable Banach space. Then
\begin{enumerate}[(i)]
\item $X$ has $\textsf{\emph{A}}_p$ if and only if there exist $\theta\in (0,1)$ and $M\in[\nn]^\omega$ such that $X$ is isomorphic to a subspace of $W^{T^*_{M,q,\theta}}_\wedge(\textsf{\emph{J}})$ if and only if there exist $\theta\in (0,1)$ and $M\in[\nn]^\omega$ such that $X$ is isomorphic to a quotient of $W^{T^*_{M,q,\theta}}_\wedge(\textsf{\emph{J}})$.
\item $X$ has $\textsf{\emph{N}}_p$ if and only if there exist $\theta\in (0,1)$ and $M\in[\nn]^\omega$ such that $X$ is isomorphic to a subspace of $W^{U^*_{M,q,\theta}}_\wedge(\textsf{\emph{J}})$ if and only if there exist $\theta\in (0,1)$ and $M\in[\nn]^\omega$ such that $X$ is isomorphic to a quotient of $W^{U^*_{M,q,\theta}}_\wedge(\textsf{\emph{J}})$.
\end{enumerate}
\end{theorem}

\begin{proof} We only detail the proof for $\textsf{{A}}_p$, as the argument for $\textsf{{N}}_p$ is similar. So assume that $X$ is a separable Banach space with $\textsf{{A}}_p$. By Theorem \ref{tedious}, there exist $\theta\in (0,1)$ and  Banach spaces $Z,Y$ with FDDs $\textsf{{F}}$, $\textsf{{H}}$, respectively, such that $X$ is isomorphic to a subspace of $Z^{T^*_{q,\theta}}_\wedge(\textsf{{F}})$, and to a quotient of $Y^{T^*_{q,\theta}}_\wedge(\textsf{{H}})$. But Proposition \ref{Schechtman} asserts that there exist $M,N \in [\Ndb]^\omega$ such that $Z^{T^*_{q,\theta}}_\wedge(\textsf{{F}})$ is isomorphic to a complemented subspace of $W^{T^*_{M,q,\theta}}_\wedge(\textsf{{J}})$ and $Y^{T^*_{q,\theta}}_\wedge(\textsf{{H}})$ is isomorphic to a complemented subspace of $W^{T^*_{N,q,\theta}}_\wedge(\textsf{{J}})$. So $X$ is isomorphic to a subspace of $W^{T^*_{M,q,\theta}}_\wedge(\textsf{{J}})$ and to a quotient of $W^{T^*_{N,q,\theta}}_\wedge(\textsf{{J}})$.

For the remaining implications we recall (Proposition \ref{modelApNp}) that $T^*_{M,q,\theta}$ has $\textsf{{A}}_p$ and that $\textsf{{A}}_p$ passes to isomorphic quotients and subspaces. 

\end{proof}

\section{Non-existence of a universal space}\label{optimal}

We conclude this article by showing that our result on small universal families is essentially optimal. 

\begin{theorem}\label{nouniversal} 
Fix $1<p\leqslant \infty$. If $U$ is any Banach space with $\textsf{\emph{N}}_p$, then there exists a Banach space $X$ with $\textsf{\emph{A}}_p$ such that $X$ is not isomorphic to any subspace of any quotient of $U$. More precisely, if $q$ is the conjugate exponent of $p$, then there exists $\theta \in (0,1)$ such that $T_{q,\theta}^*$ is not isomorphic to any subspace of any quotient of $U$.
\end{theorem}

We first recall the Schreier families, $\mathcal{S}_l$, for $k=0,1,\ldots$.   We let \[\mathcal{S}_0=\{\varnothing\}\cup \{(n):n\in\nn\},\] \[\mathcal{S}_{k+1}=\{\varnothing\}\cup \Bigl\{\bigcup_{i=1}^n E_i:n\in\nn, n\leqslant E_1<\ldots <E_n, E_i\neq \varnothing, E_i\in \mathcal{S}_k\Bigr\}.\]   We note the following associative property: For all $l,m\in\nn$, \[\mathcal{S}_{l+m}=\{\varnothing\}\cup \Bigl\{\bigcup_{i=1}^n E_i: E_i\neq \varnothing, E_i\in \mathcal{S}_l, (\min E_i)_{i=1}^n\in \mathcal{S}_m\Bigr\}.\]  We let $MAX(\mathcal{S}_k)$ denote the members of $\mathcal{S}_k$ which are maximal with respect to inclusion. 

For each $k=0,1,2,\ldots$ and each $F\in MAX(\mathcal{S}_k)$, we define $\mathbb{S}^k_F:\nn\to [0,1]$ by induction on $k$.   We first set 
\[\mathbb{S}^0_{(i)}(j)=\left\{\begin{array}{ll} 1 & : i=j \\ 0 & : i\neq j\end{array}\right.\]  Next, suppose that  $\mathbb{S}^k_E$ has been defined for each $E\in MAX(\mathcal{S}_k)$. Fix $F\in MAX(\mathcal{S}_{k+1})$.  If $F=\cup_{i=1}^n F_i$ for $n\in\nn$ and $n\leqslant F_1<\ldots <F_n$, $F_i\in  \mathcal{S}_k$, then it must be the case that $n=\min F_1$ and $F_i\in MAX(\mathcal{S}_k)$ for each $1\leqslant i\leqslant n$. We then define 
\[\mathbb{S}^{k+1}_F(j) = \left\{\begin{array}{ll} \frac{1}{n}\mathbb{S}^k_{F_i}(j) & : j\in F_i \\ 0 & : j\in \nn\setminus F \end{array}\right.\] 
It is easily checked that for $F \in MAX(\mathcal{S}_k)$, $\sum_{j\in \Ndb}\mathbb{S}^{k}_F(j)=\sum_{j\in F}\mathbb{S}^{k}_F(j)=1$.

For a Banach space $X$, $1\leqslant p\leqslant \infty$, $l\in\nn$, and $C\in [0,\infty]$, we define yet another two-player game. The confusingly named Player II chooses $m_1\in\nn$.  Player I chooses a weak neighborhood $U_1$ of $0$ in $X$,  and Player II chooses $x_1\in U_1\cap B_X$. This is the first round of the game. If $(m_1)$ is maximal in $\mathcal{S}_l$, the game terminates.  Otherwise Player II chooses $m_2\in\nn$ such that $m_1<m_2$ and $(m_1, m_2)\in \mathcal{S}_l$.  Player I then chooses a weak neighborhood $U_2$ of $0$ in $X$, and  Player II chooses $x_2\in U_2\cap B_X$.  Play continues  in this way until $m_1<\ldots<m_n$ are chosen such that $F:=(m_i)_{i=1}^n\in MAX(\mathcal{S}_l)$ and $x_1, \ldots, x_n\in B_X$ are chosen.  Player I wins if \[\Bigl\|\sum_{i=1}^n \mathbb{S}^l_F(m_i)^{1/p}x_i\Bigr\|\leqslant C\] and Player II wins otherwise. We call this the \emph{$\Phi(l,p,C)$ game on $X$}. We let $\phi_l(X,p)$ be the infimum of $C\in [0,\infty]$ such that Player I has a winning strategy in the $\Phi(l,p,C)$ game on $X$. These values need not be finite. Note also that $\phi_l(X,p) \geqslant 1$. 

\begin{proposition} Fix $1< p\leqslant \infty$. 
\begin{enumerate}[(i)]
\item For any Banach space $X$,  $\phi_1(X,p)<\infty$ if and only if $X$ has $\textsf{\emph{N}}_p$.
\item For any Banach space $X$ and for all $k,l\in\nn$,  $\phi_{k+l}(X,p) \leqslant \phi_{k}(X,p)\phi_{k+l}(X,p)$.
\item For Banach spaces $X,Y$ and $l\in\nn$, $\phi_l(X,p)\leqslant d_{\text{\emph{BM}}}(X,Y)\phi_l(Y,p)$. 
\item For a Banach space $X$ and a subspace $Y$ of $X$, $\phi_l(Y,p)\leqslant \phi_l(X,p)$ for all $l\in\nn$. 
\item For a Banach space $X$ and a subspace $Y$ of $X$, $\phi_l(X/Y,p)\leqslant 3\phi_l(X,p)$ for all $l\in\nn$.
\end{enumerate}

\label{game1}
\end{proposition}

\begin{proof}$(i)$ Let us also observe that for the $\Phi(1,p,C)$ game, a set $(m_1, \ldots, m_n)$ is maximal in $\mathcal{S}_1$ and only if $n=m_1$. Therefore in the $\Phi(1,p,C)$ game, only the initial choice of $m_1$ affects the game, since it determines how many rounds the game will have. The later values of $m_2, \ldots, m_{m_1}$ do not matter in the $l=1$ case. Moreover, if $F=(m_i)_{i=1}^n\in MAX(\mathcal{S}_1)$, then $n=m_1$ and $\mathbb{S}^1_F(m_i)^{1/p} = m_1^{-1/p}$ for all $1\leqslant i\leqslant m_1$.  Therefore the winning condition for Player I at the end of the $\Phi(1,p,C)$ game with initial choice $m_1$ is that  $\|\sum_{i=1}^{m_1} x_i\|\leqslant C m_1^{1/p}$. 

We will argue that $\textsf{n}_p(X)=\phi_1(X,p)$,  which yields $(i)$.  

If $\textsf{n}_p(X)<\infty$, then for any $C>\textsf{n}_p(X)$ and any $n\in \Ndb$, Player I has a winning strategy $\chi_n$ in the $N(n,p,C)$ game.   We will show that Player I has a winning strategy in the $\Phi(1,p,C)$ game.   If Player II begins the $\Phi(1,p,C)$ game with the choice $m_1$, then Player I plays the rest of the game according to $\chi_{m_1}$, ignoring all choices of $m_2, \ldots, m_{m_1}$. It follows from our initial observation that this is a winning strategy for Player I in the $\Phi(1,p,C)$ game. Therefore $\phi_1(X,p)\leqslant \textsf{n}_p(X)$. 

Conversely, if $\phi_1(X,p)<\infty$, then for any $C>\phi_1(X,p)$, Player I has a winning strategy $\chi$ in the $\Phi(1,p,C)$ game.  We will show that for any $n\in\nn$, Player I has a winning strategy in the $N(n,p,C)$ game.  Note that the $\Phi(1,p,C)$ game includes integers $m_i$, while the $N(n,p,C)$ game does not, so we rather than Player II will decide the values of the $m_i$'s. We choose $m_1=n$ and $m_{i+1}=n+i$ for $1\leqslant i<n$.  This results in a game with $m_1=n$ rounds, as noted in the preceding paragraph. It follows from the fact that $\chi$ is a winning strategy for the $\Phi(1,p,C)$ game that this strategy is a winning strategy in the $N(n,p,C)$ game. Therefore $\textsf{n}_p(X)\leqslant \phi_1(X,p)$. 

\smallskip
$(ii)$ Assume, as we may, that $\phi_k(X,p), \phi_l(X,p)<\infty$.  Fix $C>\phi_k(X,p)$ and $C'>\phi_l(X,p)$ and winning strategies $\chi_k$, $\chi_l$ for Player I in the $\Phi(k,p,C)$, $\Phi(l,p,C')$ games, respectively. Let $n_0=0$ and recursively choose \[m_1, U_1\cap \frac{1}{2^1}V_1, x_1, m_2, U_2\cap \frac{1}{2^2}V_1, x_2, m_3, \ldots, x_{n_1}, \] 

 \[m_{n_1+1}, U_{n_1+1}\cap \frac{1}{2^1}V_2, x_{n_1+1}, m_{n_1+2}, U_{n_1+2}\cap \frac{1}{2^2}V_2, x_{n_1+2}, m_{n_1+3}, \ldots, x_{n_2}, \] 
 
 \[ \vdots \]  \[m_{n_{r-1}+1}, U_{n_{r-1}+1}\cap \frac{1}{2^1}V_r, x_{n_{r-1}+1}, m_{n_{r-1}+2}, U_{n_{r-1}+2}\cap \frac{1}{2^2}V_r, x_{n_{r-1}+2}, m_{n_{r-1}+3}, \ldots, x_{n_r}. \]  Here, $m_i$ and $x_i$ are the choices of Player II in the $\Phi(k+l,p,CC')$ game, for which we still have to describe the strategy of Player I for choosing the weak open sets $U_{n_{i-1}+j}\cap \frac{1}{2^j}V_i$. For a given $1\leqslant i\leqslant r$, the choices  
 
 \[m_{n_{i-1}+1}, U_{n_{i-1}+1}\cap \frac{1}{2^1}V_i, x_{n_{i-1}+1}, m_{n_{i-1}+2}, U_{n_{i-1}+2}\cap \frac{1}{2^2}V_i, x_{n_{i-1}+2}, m_{n_{i-1}+3}, \ldots, x_{n_i} \]  
 are made by Player I as follows. Each set $U_{n_{i-1}+s}$ is chosen according to the strategy $\chi_k$ and we proceed until $F_i:=(m_j)_{j=n_{i-1}+1}^{n_i}$ is maximal in $\mathcal{S}_k$. This implies that 
 $$y_i:=\frac{1}{C}\sum_{j=n_{i-1}+1}^{n_i} \mathbb{S}^k_{F_i}(j-n_1-\ldots-n_{i-1})^{1/p} x_j \in B_X.$$
 Each set $V_i$ is chosen to be weak neighborhood of $0$ which is a convex symmetric subset of $W_i$, where the set $W_i$ is chosen according to the strategy $\chi_l$ in the $\Phi(l,p,C')$ game where the choices proceed as 
 \begin{align*}
 & m_1, W_1, y_1, m_{n_1+1}, W_2, y_2, \ldots, W_r,y_r.    
 \end{align*}
 Note that our construction implies that $y_i\in W_i$, which allows the strategy $\chi_l$ to apply.
 
 This game terminates when $G:=(m_{n_{i-1}+1})_{i=1}^r\in MAX(\mathcal{S}_l)$. Then $F:=\cup_{i=1}^r F_i=(m_i)_{i=1}^{n_r}\in MAX(\mathcal{S}_{k+l})$.    Moreover, \[\Bigl\|\sum_{i=1}^{n_r}\mathbb{S}^{k+l}_F(i)^{1/p}x_i\Bigr\| = C\Bigl\|\sum_{i=1}^r \mathbb{S}^l_G(i)^{1/p} \frac{1}{C}\sum_{j=n_{i-1}+1}^{n_i} \mathbb{S}^k_{F_i}(j-n_1-\ldots-n_{i-1})^{1/p}x_j\Bigr\| \leqslant CC'.\]  Player I playing in this way defines a winning strategy in the $\Phi(k+l,p,CC')$ game. Since $C>\phi_k(X,p)$ and $C'>\phi_l(X,p)$ were arbitrary, this concludes $(ii)$. 
 
\smallskip
Items $(iii)$ and $(iv)$ are clear. 

\smallskip $(v)$ Let $Y$ be a closed subspace of a Banach space $X$. Denote $Q:X\to X/Y$ the quotient map.  We shall use the following lemma, that can be found for instance in \cite{CauseyIllinois2018} (Proposition 5.6).
\begin{lemma}\label{liftweaknull} For any weak open neighborhood $V$ of $0$ in $X$, there exists a weak open neighborhood $U$ of $0$ in $X/Y$ such that $U\cap \frac13 B_{X/Y} \subset Q(B_X\cap V)$.
\end{lemma}

Let $l\in \Ndb$ and assume that $C>3\phi(X,p)$ and let $\chi_l$ be a winning strategy for Player I in the $\Phi(l,p,\frac{C}{3})$ game in $X$. We now describe a winning strategy for Player I in the $\Phi(l,p,C)$ game in $X/Y$. The players choose recursively
$$m_1,U_1,z_1,m_2,U_2,z_2,\ldots,U_r,z_r.$$
The choices of Player II are $m_i \in \Ndb$ and $z_i\in U_i\cap B_{X/Y}$, while the choices of Player I are $U_i$ weak open neighborhoods of $0$ in $X/Y$ that are convex and symmetric. In the course of the game we insert choices of $x_i \in B_X$ and $V_i$ weak open neighborhoods of $0$ in $X$ in the following order 
$$m_1,V_1,U_1,z_1,x_1,m_2,V_2,U_2,z_2,x_2\ldots,V_r,U_r,z_r,x_r.$$
We now describe the choices. Denote $V_1=\chi_l(m_1)$, then Player I  picks $U_1$, given by Lemma \ref{liftweaknull}, so that $U_1\cap \frac13 B_{X/Y} \subset Q(B_X\cap V_1)$. Next Player II picks $z_1\in U_1\cap B_{X/Y}$. The choice of $U_1$ implies the existence of $x_1\in B_X\cap V_1$ such that $Q(x_1)=\frac13 z_1$. After Player II chooses $m_2$, we pick $V_2=\chi_l(m_1,V_1,x_1,m_2)$ and Player I chooses $U_2$ so that $U_2\cap \frac13 B_{X/Y} \subset Q(B_X\cap V_2)$. The strategy of Player I should now be clear. Since $\chi_l$ is a winning strategy for Player I in the $\Phi(l,p,\frac{C}{3})$ game on $X$, when the game finishes after $r$ rounds, we have that $F:=(m_i)_{i=1}^r\in MAX(\mathcal{S}_l)$ and  
\[\Bigl\|\sum_{i=1}^r \mathbb{S}^l_F(m_i)^{1/p}z_i\Bigr\|_{X/Y} \leqslant 3\Bigl\|\sum_{i=1}^r \mathbb{S}^l_F(m_i)^{1/p}x_i\Bigr\|_X\leqslant C\] 
This finishes the proof of $(v)$. 

\end{proof}

In the next Proposition, we give a lower estimate for $\phi_l(T_{q,\theta}^*,p)$.
\begin{proposition} Fix $\theta \in (0,1)$, $1<p\le \infty$ and let $q$ be the conjugate exponent of $p$. 
\begin{enumerate}[(i)]
\item Fix $\ee,a>0$. Suppose that $I_1<\ldots <I_n$ are intervals and  $x_1<\ldots < x_m$ so that $x_i\in  aB_{T_\theta}$, and $\frac{2n}{m}<\ee$.  Then \[\frac{\theta}{m}\sum_{i=1}^n\Bigl\|I_i\sum_{j=1}^m x_j\Bigr\|_{T_\theta} \leqslant (\theta+\ee)a.\] 

\item Fix $\ee\in (0,1-\theta)$ and  $l\in\nn$. Assume that  $M=(m_i)_{i=1}^\infty,R=(r_i)_{i=1}^\infty\in [\nn]^\omega$ are such that $\theta m_1>\frac{1}{\theta}$, $m_1>\frac{2}{\ee}(1-\theta-\frac{\ee}{2})$, and    $\frac{2r_i}{m_{i+1}}<\frac{\ee}{2}$ for all $i\in\nn$, then  if $F=(m_i)_{i=1}^n\in MAX(\mathcal{S}_{2l-1})$, it follows that   \[\Bigl\|\sum_{i=1}^n \mathbb{S}^{2l-1}_F(m_i)e_{r_i}\Bigr\|_{T_\theta} \leqslant (\theta+\ee)^l,\]
where $(e_j)_{j=1}^\infty$ is the canonical basis of $T_\theta$.

\item It holds that $\phi_{2l-1}(T_{q,\theta}^*,p) \geqslant \theta^{-l/q}$. \end{enumerate}
\label{lower_estimates}
\end{proposition}

\begin{proof}$(i)$  For $1\leqslant i \leqslant n$, we let $A_i=\{j\leqslant m: \supp(x_j)\subset [\min I_i, \max I_i]\}$ and  $B_i=\{i\leqslant m:\supp(x_j)\cap I_i\neq \varnothing\}$.  Of course, $\sum_{i=1}^n |A_i|\leqslant m$ and $|A_i\setminus B_i|\leqslant 2$.    Then \begin{align*} \theta\sum_{i=1}^n \Bigl\|I_i\sum_{j=1}^m  x_j\Bigr\|_{T_\theta} \leqslant \theta \sum_{i=1}^n\sum_{j\in A_i}\|x_j\|_{T_\theta} + \sum_{i=1}^n \sum_{j\in B_i\setminus A_i}\|x_j\|_{T_\theta}  \leqslant \theta a m + 2an. \end{align*}  Therefore \[\frac{\theta}{m}\sum_{i=1}^n \Bigl\|I_i\sum_{j=1}^m x_j\Bigr\|_{T_\theta} \leqslant (\theta  + \frac{2n}{m})a\leqslant (\theta+\ee)a.\]

$(ii)$ We work by induction on $l$.  We note that \begin{align*} &\Bigl\|\sum_{i=1}^n \mathbb{S}^1_F(m_i)e_{r_i}\Bigr\|_{T_\theta} \\
&= \max\Bigl\{\Bigl\|\sum_{i=1}^n \mathbb{S}^1_F(m_i)e_{r_i}\Bigr\|_{c_0}, \theta \sup\bigl\{\sum_{j=1}^k\bigl\|I_j\sum_{i=1}^n \mathbb{S}^1_F(m_i)e_{r_i} \bigr\|_{T_\theta}: I_1<\cdots<I_k,\ (\min I_j)_{j=1}^k\in \mathcal{S}_1\bigr\}\Bigr\} \\ & \leqslant \max\Bigl\{\Bigl\|\sum_{i=1}^n \mathbb{S}^1_F(m_i)e_{r_i}\Bigr\|_{c_0}, \theta \Bigl\|\sum_{i=1}^n \mathbb{S}^1_F(m_i)e_{r_i}\Bigr\|_{\ell_1}\Bigr\} \leqslant \max\{\frac{1}{m_1}, \theta\}=\theta.\end{align*}

Next, suppose the result holds for some $l\in \Ndb$. Suppose also that $M,R$ are as in the statement and $F=(m_i)_{i=1}^n\in MAX(\mathcal{S}_{2l+1})$.      Then we can write $F=\cup_{i=1}^{t}F_i$, $F_1<\ldots <F_{t}$, $F_i\in MAX(\mathcal{S}_{2l-1})$.  For each $1\leqslant i\leqslant t$, let $G_i=\{r_j: m_j\in F_i\}$, $M_i=M\setminus \cup_{j=1}^{i-1}F_j$, and $R_i=R\setminus \cup_{j=1}^{i-1}G_j$.  We observe the convention that the empty union is the empty set, so $M_1=M$ and $R_1=R$.  Note that for each $1\leqslant i\leqslant t$, $F_i$ is the maximal initial segment of $M_i$ which is also a member of $\mathcal{S}_{2l-1}$, and the pair $(M_i, R_i)$ also satisfies the hypotheses of statement $(ii)$. So by the inductive hypothesis, for all $1\leqslant i\leqslant t$, $\|x_i\|_{T_\theta}\leqslant (\theta+\ee)^l$, where $x_i=\sum_{j=1}^\infty \mathbb{S}^{2l-1}_{F_i}(m_j)e_{r_j}$.

Next, note that since $F=\cup_{i=1}^t F_i$, $F\in MAX(\mathcal{S}_{2l+1})$, $F_1<\ldots <F_t$, and $F_i\in MAX(\mathcal{S}_{2l-1})$, it follows that $H:=(\min F_i)_{i=1}^t\in MAX(\mathcal{S}_2)$.  This means we can write $H=\cup_{i=1}^s H_i$,  where $H_1<\ldots <H_s$, $H_i\in MAX(\mathcal{S}_1)$, and $E:=(\min H_i)_{i=1}^s\in MAX(\mathcal{S}_1)$.  Since $E\in MAX(\mathcal{S}_1)$, we deduce that \[m_1=\min M=\min F=\min H= \min E = |E|=s.\] Moreover, \[x:=\sum_{i=1}^n \mathbb{S}^{2l+1}_F(m_i)e_{r_i} = \frac{1}{m_1}\sum_{i=1}^{m_1}\frac{1}{|H_i|}\sum_{j:\min F_j\in H_i} x_j.\]  For $1\leqslant i\leqslant m_1$, let $y_i=\frac{1}{|H_i|}\sum_{j:\min F_j\in H_i}x_j$, so that $x=\frac{1}{m_1}\sum_{i=1}^t y_i$. Note also that $\|y_i\|_{T_\theta}\leqslant (\theta+\eps)^l$.  Fix $I_1<\ldots <I_k$ such that $(\min I_l)_{l=1}^k\in \mathcal{S}_1$.  By omitting any $I_l$ such that $I_ly_i=0$ for all $1\leqslant i\leqslant m_1$, we can assume that $I_1y_i\neq 0$ for some $i$.   Let $i_0$ be the minimum such $i$.   Note that $k\leqslant \min I_1 \leqslant \max \supp(y_{i_0})=e_{r_{i_1}}$ for some $i_1$.  Note also that for each $i_0<i\leqslant m_1$, since $H_i\in MAX(\mathcal{S}_1)$,  $|H_i|=\min H_i=m_{i_2}$ for some $i_2>i_1$. Therefore 
\[\frac{k}{|H_i|} \leqslant \frac{r_{i_1}}{m_{i_2}} < \frac{\eps}{4}.\] 
By $(i)$ applied with $a=(\theta+\ee)^l$ and $\frac{\eps}{2}$ in place of $\eps$ it follows that for $i_0<i\leqslant m_1$,  \[\theta\sum_{l=1}^k\|I_ly_i\|_{T_\theta}=\frac{\theta}{|H_i|}\sum_{l=1}^{k}\Bigl\|I_l\sum_{j:\min F_j\in H_i}x_j\Bigr\|_{T_\theta} \leqslant (\theta+\frac{\eps}{2})(\theta+\ee)^l.\] 
Therefore 
\begin{align*}  \frac{\theta}{m_1}\sum_{l=1}^k \Bigl\|I_l \sum_{i=1}^{m_1}y_i\Bigr\|_{T_\theta} & \leqslant \frac{\theta}{m_1}\sum_{i=i_0}^{m_1}\sum_{l=1}^k \|I_ly_i\|_{T_\theta} \\ 
& \leqslant \frac{1}{m_1}\Bigl(\|y_{i_0}\|_{T_\theta} +\sum_{i=i_0+1}^{m_1}(\theta+\frac{\eps}{2})(\theta+\ee)^l\Bigr) \\ 
& \leqslant (\theta+\ee)^l\Bigl(\frac{1}{m_1}\cdot 1+\frac{m_1-1}{m_1}\cdot (\theta+\frac{\eps}{2})\Bigr) \leqslant (\theta+\ee)^{l+1}.\end{align*} Here we have used the fact that since $m_1>2(1-\theta-\frac{\eps}{2})/\ee$, 
\[\frac{1}{m_1}+\frac{m_1-1}{m_1}(\theta+\frac{\eps}{2})<\theta+\ee.\]    
We also note that \[\Bigl\|\frac{1}{m_1}\sum_{i=1}^{m_1}y_i\Bigr\|_{c_0} = \frac{1}{m_1}\max_{1\leqslant i\leqslant m_1}\|y_i\|_{c_0}  \leqslant \frac{(\theta+\ee)^l}{m_1}\leqslant \theta(\theta+\ee)^l<(\theta+\ee)^{l+1}.\]  Therefore \begin{align*} \Bigl\|\frac{1}{m_1}\sum_{i=1}^{m_1}y_i\Bigr\|_{T_\theta} & = \max\Bigl\{\Bigl\|\frac{1}{m_1}\sum_{i=1}^{m_1}y_i\Bigr\|_{c_0}, \sup\bigl\{\frac{\theta }{m_1}\sum_{l=1}^k \Bigl\|I_l\sum_{i=1}^{m_1}y_i\Bigr\|_{T_\theta}:(\min I_l)_{l=1}^k\in \mathcal{S}_1\bigr\}\Bigr\} \\ & \leqslant (\theta+\ee)^{l+1}.\end{align*}

$(iii)$ Fix $\ee\in (0,1-\theta)$. For $C<(\theta+\ee)^{-l/q}$, we construct a winning strategy for Player II in the $\Phi(2l-1,p,C)$ game on $T_{q,\theta}^*$. Recall that the canonical basis $(e_j^*)_{j=1}^\infty$ of $T_{q,\theta}^*$ is normalized and weakly null.   First Player II chooses $m_1$ so large that $\theta m_1>1$ and $m_1>\frac{2}{\ee}(1-\theta-\frac{\ee}{2})$.   Player I chooses some weak neighborhood $U_1$ of $0$ in $T_{q,\theta}^*$. Player II then chooses some $r_1$ so large that $e_{r_1}^*\in U_1$.    Assuming $m_1, U_1, r_1, \ldots, m_j, U_j, r_j$ have been chosen and $(m_i)_{i=1}^j\in \mathcal{S}_{2l-1}\setminus MAX(\mathcal{S}_{2l-1})$, Player II chooses some $m_{j+1}$ so large that $\frac{r_j}{m_{j+1}}<\frac{\ee}{4}$ and $m_{j+1}>m_j$.    Player I then chooses $U_{j+1}$. Player II chooses $r_{j+1}>r_j$ such that $e^*_{r_{j+1}}\in U_{j+1}$.      When the game terminates at some $F=(m_i)_{i=1}^n\in MAX(\mathcal{S}_{2l-1})$, we arbitrarily choose $m_{n+1}, r_{n+1}, m_{n+2}, r_{n+2}, \ldots$ such that $r_n<r_{n+1}<\ldots$, $m_n<m_{n+1}<\ldots$, and  $\frac{r_i}{m_{i+1}}<\frac{\ee}{4}$ for all $i\in\nn$.    By $(ii)$, 
\[\Bigl\|\sum_{i=1}^n \mathbb{S}^{2l-1}_F(m_i)^{1/q}e_{r_i}\Bigr\|_{T_{q,\theta}} = \Bigl\|\sum_{i=1}^n \mathbb{S}^{2l-1}_F(m_i)e_{r_i}\Bigr\|_{T_\theta}^{1/q} \leqslant (\theta+\ee)^{l/q}.\] 
Therefore 
\begin{align*}
\Bigl\|\sum_{i=1}^n \mathbb{S}^{2l-1}_F(m_i)^{1/p}e_{r_i}^*\Bigr\|_{T_{q,\theta}^*} &\geqslant (\theta+\ee)^{-l/q}\Bigl(\sum_{i=1}^n \mathbb{S}^{2l-1}_F(m_i)^{1/p}e_{r_i}^*\Bigr)\Bigl(\sum_{i=1}^n \mathbb{S}^{2l-1}_F(m_i)^{1/q}e_{r_i}\Bigr)\\
&=(\theta+\ee)^{-l/q}\sum_{i=1}^n \mathbb{S}^{2l-1}_F(m_i) =(\theta+\ee)^{-l/q}.  
\end{align*}   
This shows that the strategy outlined above is a winning strategy for Player II in the $\Phi(2l-1,p,C)$ game.    Since we can do this for any $C<(\theta+\ee)^{-l/q}$ and  $0<\ee<1-\theta$, $\phi_{2l-1}(T_{q,\theta}^*,p) \geqslant \theta^{-l/q}$. 
\end{proof}

\begin{proof}[Proof of Theorem \ref{nouniversal}] 
Assume $U$ is an infinite dimensional Banach space which has $\textsf{N}_p$. Then by  items $(i)$ and $(ii)$ of Proposition \ref{game1}, $a=\phi_1(U)\in [1,\infty)$ and $\phi_l(U)\leqslant a^l$, for all $l\in \Ndb$. Fix now $\theta\in (0,1)$ such that $\theta^{-1/q}>a^2$ and let $X=T_{q,\theta}^*$ (remember that $X$ has $\textsf{A}_p$). If it were true that $X$ is isomorphic to a subspace of a quotient of $U$, then by items $(iii)$, $(iv)$ and $(v)$ of Proposition \ref{game1}, there would exist a constant $C$ such that for all $l\in\nn$,  
\[(\theta^{-1/q})^l \leqslant \phi_{2l-1}(X,p) \leqslant C\phi_{2l-1}(U,p) \leqslant C(a^2)^l.\]But this is impossible, since $\theta^{-1/q}>a^2$. 
\end{proof}

%------------------- Biblio -------------------
\bibliographystyle{amsplain}

\begin{bibsection}
\begin{biblist}

%\bibitem{AGM2016}




\bib{Bossard1994}{article}{
   author={Bossard, B.},
   title={Théorie descriptive des ensembles en géométrie des espaces de Banach},
   journal={Thèse de doctorat de Mathématiques de l'Université Paris VI},
   volume={},
   date={1994},
   number={},
   pages={}
}

\bib{Bossard2002}{article}{
   author={Bossard, B.},
   title={A coding of separable Banach space. Analytic and coanalytic families of Banach spaces},
   journal={Fundam. Math.},
   volume={172},
   date={2002},
   number={},
   pages={117--151}
}


\bib{CasazzaShura}{book}{
   author={Casazza, P.G.},
   author={Shura, T.J.},
   title={Tsirelson's space. {With} an appendix by {J}. {Baker}, {O}. {Slotterbeck} and {R}. {Aron}},
   Series = {Lect. Notes Math.},
   journal={},
   volume={1363},
   year={1989},
   number={},
   Publisher = {Berlin etc.: Springer-Verlag},
}


\bib{CauseyThesis}{article}{
   author={Causey, R. M.},
   title={Szlenk index, upper estimates and embedding in Banach spaces},
   journal={PhD Dissertation, Texas A\&M University},
   volume={},
   date={2014},
   number={},
   pages={}
}



\bib{CauseyIllinois2018}{article}{
   author={Causey, R. M.},
   title={Concerning $q$-summable Szlenk index},
   journal={Illinois J. Math.},
   volume={62},
   date={2018},
   number={1-4},
   pages={381--426}
}





%\bibitem{CauseyPositivity2018}

\bib{CauseyPositivity2018}{article}{
   author={Causey, R. M.},
   title={Power type asymptotically uniformly smooth and asymptotically
   uniformly flat norms},
   journal={Positivity},
   volume={22},
   date={2018},
   number={5},
   pages={1197--1221}
}

%\bibitem{Causey3.5}

\bib{Causey3.5}{article}{
   author={Causey, R. M.},
   title={Three and a half asymptotic properties},
   journal={Studia Math.},
   volume={257},
   date={2021},
   number={2},
   pages={155--212}
}

\bib{CauseyFovelleLancien}{article}{
author={Causey, R. M.},
author={Fovelle, A.},
author={Lancien G.},
title={Asymptotic smoothness in Banach spaces, three space properties and applications},
journal={Preprint arxiv: https://arxiv.org/abs/2110.06710}
}


\bib{CauseyLancienT_p}{article}{
author={Causey, R. M.},
author={Lancien G.},
title={Universality, complexity and asymptotically uniformly smooth Banach spaces},
journal={Preprint arxiv : https://arxiv.org/abs/2203.13128}
}




	\bib{CauseyNavoyan}{article}{
				author={Causey, R. M.},
				author={Navoyan, K. V.},
				title={Factorization of Asplund operators},
				journal={J. Math. Anal. Appl.},
				volume={479},
				date={2019},
				number={1},
				pages={1324--1354}
			}
			
			
			
			
\bib{DFJP}{article}{
   author={Davis, W. J.},
   author={Figiel, T.},
   author={Johnson, W. B.},
   author={Pe{\l}czy{\'n}ski, A.},
   title={Factoring weakly compact operators},
   journal={J. Funct. Anal.},
   volume={17},
   date={1974},
   number={},
   pages={311--327}
}			
			

%\bibitem{DKLR2017}



\bib{FOSZ}{article}{
				author={Freeman, D.},
				author={Odell, E.},
				author={Schlumprecht, Th.},
				author={Zs\'ak, A.},
				title={Banach spaces of bounded Szlenk index. II},
				journal={Fundam. Math.},
				volume={205},
				date={2009},
				number={2},
				pages={162--177}
			}
			
			
			
\bib{GKL2001}{article}{
				author={Godefroy, G.},
				author={Kalton, N. J.},
				author={Lancien, G.},
				title={Szlenk indices and uniform homeomorphisms},
				journal={Trans. Amer. Math. Soc.},
				volume={353},
				date={2001},
				number={10},
				pages={3895--3918}
			}
			
			

\bib{HajekLancien2007}{article}{
   author={H\'ajek, P.},
   author={Lancien, G.},
   title={Various slicing indices on Banach spaces,},
   journal={Med. J. of Math.},
   volume={4},
   date={2007},
   number={},
   pages={179--190}
}

\bib{JohnsonZippin}{article}{
				author={Johnson, W.B.},
				author={Zippin, M.},
				title={Subspaces and quotient spaces of \((\sum G_n)_{l_p}\) and \((\sum G_n)_{c_0}\)},
				journal={Isr. J. Math.},
				volume={17},
				date={1974},
				number={},
				pages={50--55}
			}
			
			
			
\bib{Lancien2006}{article}{
				author={Lancien, G.},
				title={A survey on the Szlenk index and some of its applications},
				journal={RACSAM. Rev. R. Acad. Cienc. Exactas Fís. Nat. Ser. A Mat.},
				volume={100},
				date={2006},
				number={1-2},
				pages={209--235}
			}
			
\bib{OS2002}{article}{
				author={Odell, E.},
				author={Schlumprecht, Th.},
				title={Trees and branches in {Banach} spaces},
				journal={Trans. Am. Math. Soc.},
				volume={354},
				year={2002},
				number={10},
				pages={4085--4108}
			}			
			
			
			
			
\bib{OSZ2007}{article}{
				author={Odell, E.},
				author={Schlumprecht, Th.},
				author={Zs{\'a}k, A},
				title={A new infinite game in {Banach} spaces with applications},
				BookTitle = {Banach spaces and their applications in analysis. Proceedings of the international conference, Miami University, Oxford, OH, USA, May 22--27, 2006. In honor of Nigel Kalton's 60th birthday},
				journal={},
				volume={},
				year={2007},
				number={},
				Publisher = {Berlin: Walter de Gruyter},
				pages={147--182}
			}






\bib{Schechtman1975}{article}{
 Author = {Schechtman, G.},
 Title = {{On Pelczynski's paper 'Universal bases'}},
 journal = {{Isr. J. Math.}},
 Volume = {22},
 Pages = {181--184},
 Year = {1975},
}


\end{biblist}

\end{bibsection}

\end{document}